\documentclass[11pt]{amsart}

\usepackage{mathrsfs,graphicx,latexsym,tikz,color,euscript,caption,multirow}
\usepackage{amsfonts,amssymb,amsmath,amscd,amsthm}
\usepackage{epstopdf}
\usepackage{hyperref}
\usepackage{upgreek}
\usepackage{comment}
\usepackage{pdfsync}

\usepackage{color}
\newcommand{\R}{\mathbb R}

\newcommand{\N}{\mathbb N}

\newcommand{\bfr}{\mathbf{r}}

\newcommand{\bfv}{\mathbf{v}}  
  
\newcommand{\action}{\mathcal A}

\newcommand{\rt}{\tilde{\mathbf{r}}}
\newcommand{\bfu}{\mathbf{u}}

\theoremstyle{plain}

\newcommand{\rr}{\mathbb R}

\newcommand{\A}{\mathcal A}
\newcommand{\Ah}{\mathcal{A}_{\mathbf{h}}}
\newcommand{\Ap}{\mathcal{A}_{\mathbf{p}}}

\newcommand{\D}{\mathcal D}
\newcommand{\bs}{\mathbf s}
 
\newcommand{\al}{\alpha}

\newcommand{\gm}{\gamma}

\newcommand{\ep}{\varepsilon}
\newcommand{\Om}{\Omega}
\newcommand{\dl}{\delta}

\newcommand{\lmd}{\lambda}

\newcommand{\vph}{\varphi}

\newcommand{\vn}{\vph_n}

\newcommand{\ey}{\frac{1}{2}}
\newcommand{\se}{\frac{2}{3}}
\newcommand{\sy}{\frac{1}{3}} 

\newcommand{\qt}{\tilde{q}}

\newcommand{\xb}{\bar{x}}

\newcommand{\xs}{x^*_0}

\newcommand{\la}{\langle}
\newcommand{\ra}{\rangle}

\newcommand{\ql}{q^{\lambda}}

\theoremstyle{plain}
\newtheorem{thm}{Theorem}[section]
\newtheorem{prop}{Proposition}[section]
\newtheorem{lem}{Lemma}[section]

\newtheorem{rem}{Remark}[section]

\theoremstyle{definition}
\newtheorem{dfn}{Definition}[section]
\newtheorem{asm}{Assumption}

\begin{document}

\title[Marchal-type lemmas]{Marchal-type lemmas for the restricted $(N+1)$-body problem with their applications}
 

\author{Kuo-chang Chen}
\address{Department of Mathematics, National Tsing Hua University, Hsinchu, Taiwan  \\
 \and Institute of Mathematics, Academia Sinica, Taipei, Taiwan}
 \email{kchen@math.nthu.edu.tw}

\author{Guowei Yu}
\address{Chern Institute of Mathematics and LPMC, Nankai University, Tianjin, China}
\email{yugw@nankai.edu.cn}

\thanks{K.C. is  supported in parts by the National Science and Technology Council (grant number NSTC 113-2115-M-007-012-MY3), and Academia Sinica (grant number AS-ASSA-113-02). G. Y. is supported by  NSFC (No. 12671224), Nankai Zhide Fundation and the Fundamental Research Funds for the Central Universities. 
}

\begin{abstract}
    The direct methods of calculus of variations have had a great success in the classical Newtonian $N$-body problem in this century. 
    First in finding periodic solutions and more recently in the existence of parabolic, hyperbolic and hyperbolic-parabolic solutions. 
    The foundation of many of these results is Marchal's lemma, which guarantees fixed-end local minimizers are collision-free and hence 
    classical solutions of the $N$-body problem. Meanwhile almost no result is available in the restricted problem through the direct 
    methods of calculus of variations, mainly due to the lack of a Marchal-type lemma. 

    In this paper such a lemma and its variation will be obtained for the restricted problem. We then apply them to the restricted problem to study unbounded (hyperbolic, parabolic, bi-hyperbolic and bi-parabolic) solutions, and bounded (periodic and quasi-periodic) solutions. We emphasize that the extension to the restricted case does not follow from the standard argument for the unrestricted case. 
\end{abstract}

\maketitle

\section{Introduction }\label{sec:intro}

The motion of Newtonian $N$-body problem in $\rr^d$ ($d \ge 2$) can be described by
\begin{equation}
\ddot{q}_i(t) = -\sum_{ j \ne i, j=1 }^{N}\frac{m_j (q_i(t) - q_j(t))}{|q_i(t) - q_j(t)|^3}, \;\;  \forall i = 1, 2, \dots, N \label{eq;Nbody}
\end{equation} 
where $m_i>0$ is the mass of the $i$-th body and $q_i(t) \in \rr^d$ represents its trajectory.

It is well-known solutions of the $N$-body problem correspond to collision-free critical points of the following action functional 
$$ \mathbb{A}(q; t_1, t_2) := \int_{t_1}^{t_2} \ey \sum_{i =1}^N m_i|\dot{q_i}|^2 + \sum_{1 \le i < j \le N} \frac{m_i m_j}{|q_i -q_j|} \,dt
$$ 

For a long time the application of the direct methods of calculus of variations to the Newtonian $N$-body problem was 
obstructed by the possibility of collision in action minimizers. Change began with the proofs of the Figure-Eight solution by Chenciner and Montgomery \cite{CM00} and the Hip-Hop solution by Chenciner and Venturelli \cite{CV00}, where the 
solutions were found as collision-free action minimizers under suitable symmetric constraints. 

Shortly afterwards, based on an averaging argument by Marchal \cite{Mc02}, the following result now 
known as \emph{Marchal's lemma} 
was proven by Chenciner for the Newtonian $N$-body problem in \cite{C02}, which rules out collision in fixed-end local minimizers.

\begin{thm}[Marchal's lemma] \label{thm;Marchal} If $q \in H^1_{t_1, t_2}(\rr^{Nd})$ is a fixed-end local minimizer of $\mathbb{A}$; 
i.e., there is an $\ep>0$, such that $\mathbb{A}(q; t_1, t_2) \le \mathbb{A}(p; t_1, t_2)$, for all $p \in H^1_{t_1, t_2}(\rr^{Nd})$ satisfying $p(t_i)=q(t_i)$, $i =1, 2$ and $\|p-q\|_{H^1} < \ep$,  then $q(t)$ is collision-free for all $t \in (t_1, t_2)$. \end{thm}
Later it was generalized by Ferrario and Terracini to general homogeneous potentials and equivariant problem in  \cite{FT04}. 

Since then the direct methods of calculus of variations have had a lot of success in the $N$-body problem, first in periodic solutions, see \cite{FT04}, \cite{TV07}, \cite{Ch08}, \cite{Yu21} and references within, then in parabolic solutions \cite{MV09}, hyperbolic solutions \cite{MV20}, and hyperbolic-parabolic solutions \cite{PT24}. In most of these works, Marchal's lemma played a fundamental role. 

If we introduce an additional body whose mass is much smaller than the given $N$ bodies, we may approximately 
assume the gravitational force of the newly introduced body (to be referred as \emph{the massless body}) is too small that 
it does not affect the motion of the $N$ large masses (to be referred as \emph{the primaries}). 
Then we obtain the so-called \emph{restricted $(N+1)$-body problem}, which studies the motion of the massless body under the 
gravitational force of the $N$ primaries, whose trajectories are given in advance, and can be described by the following equation
\begin{equation}
\label{eq;RNBP1} \ddot{x}(t) = \nabla_x U(x(t), t) =- \sum_{i=1}^N \frac{m_i(x(t) - q_i(t))}{|x(t) - q_i(t)|^3}, 
\end{equation}
where $U(x, t)$ is the time-dependent potential function
\begin{equation}
\label{eq;Uzt} U(x, t) = \sum_{i=1}^N \frac{m_i}{|x- q_i(t)|}.
\end{equation}

The restricted $(N+1)$-body problem is a good model in describing the motions of asteroids, comets and satellites in gravitational systems, 
and has been studied by many authors including Poincar\'e \cite{Poincare1993} and Birkhoff \cite{Birkhoff1966}. 
Well-known examples include the circular restricted three body problem and the elliptic restricted three body problem, where the motions of the two primaries are assumed to be circular and elliptic periodic solutions of the two body problem. 

Like the $N$-body problem, the restricted problem also has a variational formalism, and its solutions correspond to collision-free critical points of the following action functional
\begin{equation}
\label{eq;A-restricted} \A(x; t_1, t_2) = \int_{t_1}^{t_2} \ey |\dot{x}(t)|^2 + U(x(t), t) \,dt. 
\end{equation}
Despite this so far very few results were obtained through variational methods in the restricted problem. 
The only ones that we are aware of are in \cite{Terr00}, \cite{Sh19} and \cite{PT22}, where in the last two references the authors consider a rather special restricted $3$-body problem where the masses of the two primaries are equal and their motions are homothetic collision-ejection solutions with negative energy. Such a setting makes it relatively easy to rule out collision between the massless body and the primaries in a variational approach.

One of the reasons that variational methods are much less successful in the restricted problem is the lack of a Marchal-type lemma. 
The purpose of our paper is to prove such lemmas and demonstrate their potential in the restricted problem by proving the existence of several different types of bounded and unbounded solutions based on these lemmas. 

The extension of Marchal's lemma to the restricted problem does not follow from the standard argument for the unrestricted case. In fact, there is a mistake in the preprint \cite{Yu22} by the second author, where he thought the standard argument could work. The place where it fails is regarding the isolation of collision moments. Without this one does not have Sundman's estimates \cite{Sp69} near a collision moment, and then the blow-up argument and Marchal's average method followed up can not be used.  More details of this can be found in Section \ref{sec:local_minimizers_are_collision_free}.

\section{Main results} \label{sec:results}

Our first result is the following theorem, which rules out collision in fixed-end local minimizers. 
It can be seen as an analogue of Marchal's lemma for the restricted case. 
\begin{thm}
\label{thm;Marchal-restricted} Given any two times $t_1 <t_2$. If $q|_{[t_1, t_2]}$ is collision-free, for all $t \in [t_1, t_2]$, and $x \in H^1_{t_1, t_2}(\rr^d)$ is a fixed-end local minimizer of $\A$, i.e., there is an $\ep>0$, such that 
$$ \A(x; t_1, t_2) \le \A(y; t_1, t_2), \;\; \forall y \in H^1_{t_1, t_2}(\rr^d) \text{ with }  \|y-x\|_{H^1} < \ep \text{ and } y(t_i)=x(t_i), \; i =1, 2$$
then $x(t)$ is collision-free, for all $t \in (t_1, t_2)$, and a classical solution of \eqref{eq;RNBP1}.
\end{thm}
\begin{rem}
    Notice that one can choose to start or/and end at a collision, so we can not rule out the possibility of collision at the boundary moments.
\end{rem}

For applications, particularly in finding periodic and quasi-periodic solutions with symmetries, it will be useful to have certain generalizations of the Marchal lemma to minimizing problems with free boundaries. For the unrestricted problem, such generalizations can be found in \cite{Ch06} by the first author. Our second result is the analogues for the restricted problem.  

\begin{dfn}
\label{dfn;free-boundary} Given any two different times $t_1 < t_2$. For $j =1, 2$, let $X_j$ be one of the following four closed subsets of $\rr^d$: 
\begin{enumerate}
    \item a proper linear subspace of $\rr^d$ with $1 \le \text{dim}(X_j) < d$;
    \item a straight line segment in $\rr^d$ connecting $q_{i_1}(t_j)$ and $q_{i_2}(t_j)$ with $1 \le i_1 \ne i_2 \le N$;
    \item a straight ray in $\rr^d$ starting from $q_i(t_j)$ with $i =1, \dots, N$;
    \item a point in $\rr^d$,
\end{enumerate}
we define the set of admissible paths as
$$ B_{t_1, t_2}(X_1, X_2) = \{ x \in H^1_{t_1, t_2}(\rr^d): x(t_1) \in X_1, x(t_2) \in X_2 \}.$$
\end{dfn}

\begin{thm}
\label{thm;Marchal-rest-2} If $x|_{[t_1, t_2]}$ is a local minimizer of $\A$ in $B_{t_1, t_2}(X_1, X_2)$, i.e., $\exists \ep>0$, such that 
$$\A(x; t_1, t_2) \le \A(y; t_1, t_2), \;\; \forall y \in B_{t_1, t_2}(X_1, X_2) \text{ satisfying } \|y-x\|_{H^1} < \ep,
$$
then $x(t)$ is collision-free, for all $t \in (t_1, t_2)$ and at the boundaries 
\begin{enumerate}
\item[(a)] if $X_j$, $j =1, 2$, satisfies condition $(1)$ in Definition \ref{dfn;free-boundary}, then $x(t_j)$ is collision-free with $\dot{x}(t^{+}_1)$ perpendicular to $X_1$ and $\dot{x}(t^-_{2})$ perpendicular to $X_2$;
\item[(b)] if $X_1$ satisfies condition $(2)$ in Definition \ref{dfn;free-boundary}, either $x(t_1)$ is collision-free with $\dot{x}(t_1^+)$ perpendicular to $X_1$, or $x(t_1) = q_{i_k}(t_1)$, for $k =1, 2$, with 
$$  \lim_{t \to t_1^{+}} \frac{x(t)- q_{i_k}(t)}{|x(t)- q_{i_k}(t)|} = -\frac{ q_{i_{3-k}}(t_1)-q_{i_k}(t_1) }{|q_{i_{3-k}}(t_1) -q_{i_k}(t_1)|}; 
$$
\item[(c)] if $X_2$ satisfies condition $(2)$ in Definition \ref{dfn;free-boundary}, either $x(t_2)$ is collision-free with $\dot{x}(t_2^-)$ perpendicular to $X_2$, or $x(t_2) = q_{i_k}(t_2)$, for $k =1, 2$, with
$$  \lim_{t \to t_2^{-}} \frac{x(t)- q_{i_k}(t)}{|x(t)- q_{i_k}(t)|} = -\frac{q_{i_{3-k}}(t_2)-q_{i_k}(t_2)}{|q_{i_{3-k}}(t_2)-q_{i_k}(t_2)|}; 
$$
\item[(d)] if $X_1$ satisfies condition $(3)$ in Definition \ref{dfn;free-boundary}, either $x(t_1)$ is collision-free with $\dot{x}(t_1^+)$ perpendicular to $X_1$, or $x(t_1) = q_i(t_1)$ with
$$   \lim_{t \to t_1^{+}} \frac{x(t)- q_{i}(t)}{|x(t)- q_{i}(t)|} = -\frac{p- q_{i}(t_1)}{|p- q_{i}(t_1)|}, \;\; \forall p \in X_1 \setminus \{q_i(t_1)\}; $$
\item[(e)] if $X_2$ satisfies condition $(3)$ in Definition \ref{dfn;free-boundary}, either $x(t_2)$ is collision-free with $\dot{x}(t_2^-)$ perpendicular to $X_2$, or $x(t_2) = q_i(t_2)$ with
$$   \lim_{t \to t_2^{-}} \frac{x(t)- q_{i}(t)}{|x(t)- q_{i}(t)|} = -\frac{p- q_{i}(t_2)}{|p- q_{i}(t_2)|}, \;\; \forall p \in X_2 \setminus \{q_i(t_2)\}. $$
\end{enumerate}
\end{thm}

\begin{rem}
    From the minimizing property of the Kepler problem (see Proposition \ref{prop;rbar-not-minimizer-2}), it is impossible to rule out collisions completely at the boundaries in the above theorem besides in property (a) using local deformation methods. Depending on the choices of masses it might be possible to rule out some of these boundary collisions using level estimates, we will not explore in this direction as it is not the main purpose of our paper.   

    Meanwhile binary collisions are regularizable in both restricted and unrestricted problems (see \cite{LC20}, \cite{Moser70} and \cite{SS71}). However for a collision-ejection solution to be regularizable, we need both one-sided limits of the collision direction and the energy to exist and coincide when the time parameter approaches the collision moment from both sides (see Proposition \ref{prop;regularizble}). Because of this the results about one-sided limits in the above theorem are important, as with them we will be able to show these minimizers can be extended to regularizable solutions through proper symmetries even when there are collisions at the boundaries.  
\end{rem}

Now we shall demonstrate the potential of the above Marchal-type lemmas in the restricted problem. For this, recall that when the motions of the primaries are bounded, i.e., 
\begin{equation} \label{eq;r0}
    r_0=\sup \{ |q_i(t)|:\; \forall t \in \rr \text{ and } i = 1, \dots, N \} < \infty
\end{equation}
 the final motion of the massless body has the following four possibilities (see Chazy \cite{Chazy22}):  
\begin{itemize}
\item \textbf{hyperbolic}: $|x(t)| \to \infty$ and $|\dot{x}(t)| \to c > 0$, as $t \to \pm \infty$;
\item \textbf{parabolic}: $|x(t)| \to \infty$ and $|\dot{x}(t)| \to  0$, as $t \to \pm \infty$;
\item \textbf{bounded}: $\limsup_{t \to \pm \infty} |x(t)| < \infty$;
\item \textbf{oscillatory}: $\limsup_{t \to \pm \infty} |x(t)| =\infty$ and $\liminf_{t  \to \pm \infty} |x(t)| < \infty$, 
\end{itemize} 

Using Theorem \ref{thm;Marchal-restricted} and the approach from \cite{PT24}, we first prove the following two results about the existence of hyperbolic and parabolic solutions for the restricted problem. 
\begin{thm}
\label{thm;hyperbolic} When the motion of the primaries are bounded and collision-free, given any initial time $t_0 \in \rr$, initial configuration $x_0 \in \rr^d$ and asymptotic direction $\bs \in \mathbb{S}: = \{ x \in \rr^d: |x| =1 \}$,
 for any $h>0$, there is a hyperbolic solution $x|_{[t_0, +\infty)}$ satisfying $x(t_0) =x_0$ and 
    \begin{equation}
        \label{eq;hyp} x(t) = \sqrt{2h}t\bs + \frac{m}{2h} (\ln t) \bs + o(\ln t), \text{ as  } t \to +\infty, \; \text{ where } m = \sum_{i=1}^N m_i.
    \end{equation}
A similar result holds for $t \to -\infty$.  
\end{thm}

\begin{thm}
\label{thm;parabolic}  When the motion of the primaries are bounded and collision-free, given any initial time $t_0 \in \rr$, initial configuration $x_0 \in \rr^d$ and asymptotic direction $\bs \in \mathbb{S}$, there is a parabolic solution $x|_{[t_0, +\infty)}$ satisfying $x(t_0) =x_0$ and
    \begin{equation}
         \label{eq;para} x(t) = \al t^{\se} \bs  + O(t^{\ey}), \text{ as } t \to +\infty, \text{ where } \al =\left( 9m/2 \right)^{\frac{1}{3}}.
     \end{equation} 
A similar result holds for $t \to -\infty$.  
\end{thm}

While the above results are about hyperbolic and parabolic solutions with given initial configuration and time, it seems to be even more interesting to consider the existence of \textbf{bi-hyperbolic} or \textbf{bi-parabolic solutions} with given  asymptotic behaviors when time goes to both positive and negative infinity, as they are closely related to the scattering theory in particle systems. So far there are very few results for these solutions in celestial mechanics for either the restricted problem (\cite{PT22}) or the unrestricted problem (\cite{DMMY20} and \cite{Mk26}) . 

With Theorem \ref{thm;Marchal-rest-2}, we can obtain some partial results, when the motion of the primaries satisfies certain symmetries. To explain them, let $V$ be a proper linear subspace of $\R^d$, and let $R_V$ be the reflection map across $V$.
We say a curve $\gamma$ in $\R^d$ is symmetric with respect to $V$ at time $t_0\in\R$, if  
$$\gamma(t) \;=\; R_V \gamma(2t_0 -t), \quad \forall t\in \R.$$
This concept can be naturally extended to curves $\gamma=(\gamma_1,\cdots,\gamma_N)$ in $(\R^d)^N$ by setting 
$$R_V \gamma = (R_V\gamma_1,\cdots,R_V\gamma_N). $$ 

\begin{dfn}
    \label{dfn;V-sym} We say $\gm (t) =(\gm_i(t))_{i =1}^N$ is symmetric with respect to $V$ at time $t_0$, if 
    $$ \gm(t) = R_V \gm(2t_0-t), \quad \forall t \in \rr. $$    
\end{dfn}

\begin{asm} \label{asm;bi-hyp-para}
    $q|_{\rr}$ is bounded, collision-free and symmetric with respect to a one-dim linear subspace $V$ of $\rr^d$ at time $t_0$.
\end{asm}
Under the above assumption $q_i(t_0)$, $i =1, \dots, N$, divide $V$ into $N+1$ connected intervals. We denote the closure of each interval by $\Om_i(V)$, $i =0, 1, \dots, N$ with $\Om_0(V)$ and $\Om_N(V)$ representing the two unbounded intervals.

\begin{thm}
    \label{thm;bi-hyp-para} Under Assumption \ref{asm;bi-hyp-para}, for any asymptotic direction $\bs \in \mathbb{S}$ and any $h>0$, and $\Om_i(V)$, $i = 1, \dots, N-1$, 
    \begin{enumerate}
         \item[(a)] there is at least one bi-hyperbolic solution $x: \rr \to \rr^d$ satisfying $x(t_0) \in \Om_i(V)$, $x(t) = R_V x(2t_0-t)$, for all $t \in \rr$ and the asymptotic expression \eqref{eq;hyp}, as $t \to +\infty$; 
         \item[(b)] there is at least one bi-parabolic solution $x: \rr \to \rr^d$ satisfying $x(t_0) \in \Om_i(V)$, $x(t) = R_V x(2t_0-t)$, for all $t \in \rr$ and the asymptotic expression \eqref{eq;para}, as $t \to +\infty$.
     \end{enumerate} 
     Each of these solutions is either collision-free or regularizable with a single collision at time $t_0$. 
\end{thm}
\begin{rem}
    From the proof we will see for different $\Om_i(V)$ the corresponding bi-hyperbolic or bi-parabolic solutions are distinct.
\end{rem}

Besides its connection with scattering theory as we already mentioned, the above results are interesting, as the existence of bi-parabolic solutions played a key role in proving oscillatory solutions in celestial mechanics for both the restricted problem 
(see \cite{St60},  \cite{Alk72}, \cite{MG73}, \cite{Moser73}, \cite{LS80b}, \cite{GMS16}, \cite{GMSS17}, \cite{PT22}) and the unrestricted problem (see \cite{LS80a}, \cite{Mk84} and \cite{GMPT22}).

While the above results are all about unbounded solutions, now we will study bounded solutions, in particular periodic and quasi-periodic solutions. Since the restricted problem is time-dependent, we need the motion of the primaries to be periodic or quasi-periodic as well. One way to obtain this is through symmetry and for the next three theorems we shall always assume the motion of the primaries satisfies the following assumption. 

\begin{asm} \label{asm;periodic}
    $q|_{\rr}$ is collision-free, and for $j =1, 2$, there is a proper linear subspace  $V_j$ of $\R^d$ and a moment $t_j$ ($t_1 < t_2$), such that $q|_{\rr}$ is symmetric with respect to $V_j$ at the moment $t_j$. 
\end{asm}

\begin{thm}
     \label{thm;periodic-1} Under Assumption \ref{asm;periodic}, when $V_1 \cap V_2 =\{0\}$, $\A$ has at least one minimizer $x|_{[t_1, t_2]}$ in $B_{t_1, t_2}(V_1, V_2)$, which is collision-free for all $t \in [t_1, t_2]$ with $\dot{x}(t_1^+)$ perpendicular to $V_1$ and $\dot{x}(t_2^-)$ perpendicular to $V_2$. In particular, $x|_{[t_1, t_2]}$ can be extended to a collision-free solution $x|_{\rr}$ of \eqref{eq;RNBP1} through the symmetries with respect to $V_1$ and $V_2$. 
 \end{thm} 

While the solution obtained in the above theorem is always collision-free, it only gives us one such solution. Meanwhile when at least one of the dimensions of $V_1$ and $V_2$ is one, we are able to find more bounded solutions but with the possibility of containing regularizable collisions. 

 \begin{thm}
     \label{thm;perodic-2} Under Assumption \ref{asm;periodic}, when $\text{dim}(V_1) = 1$, for any $i = 0, 1, \dots, N$, if $\Om_i(V_1) \cap V_2$ is empty or bounded in $\rr^d$, then $\A$ has at least one minimizer $x|_{[t_1, t_2]}$ in $B_{t_1, t_2}(\Om_i(V_1), V_2)$, which is collision-free, for all $t \in (t_1, t_2]$, with $\dot{x}(t_2^-)$ perpendicular to $V_2$. 

     Meanwhile $x(t_1)$ is either collision-free with $\dot{x}(t_1^+)$ perpendicular to $\Om_i(V_1)$, or it collides with one of the boundaries of $\Om_i(V_1)$ (when $i =0$ or $N$, the boundary of $\Om_i(V_1)$ contains only one point). In particular, $x|_{[t_1, t_2]}$ can be extended to a collision-free or regularizable solution $x|_{\rr}$ of \eqref{eq;RNBP1} through the symmetries with respect to $V_1$ and $V_2$. 
 \end{thm}

    \begin{rem}
        Notice that there is a similar result, when $\text{dim}(V_2)=1$. 
    \end{rem}

    When the dimensions of $V_1$ and $V_2$ are both one, we can obtain even more bounded solutions, but with the possibility of containing more regularizable collisions. 

\begin{thm}
\label{thm;periodic-3} Under Assumption \ref{asm;periodic}, when $\text{dim}(V_1) = \text{dim}(V_2)=1$, for any $0 \le i \ne j \le N$, if $ \Om_i(V_1) \cap \Om_j(V_2)$ is empty or bounded in $\rr^d$, then $\A$ has at least one minimizer $x|_{[t_1, t_2]}$ in $B_{t_1, t_2}(\Om_i(V_1), \Om_j(V_2))$, which is collision-free, for all $t \in (t_1, t_2)$. 

Meanwhile $x(t_1)$ is either collision-free with $\dot{x}(t_1^+)$ perpendicular to $\Om_i(V_1)$ or it collides with one of the boundaries of $\Om_i(V_1)$, and $x(t_2)$ is either collision-free with $\dot{x}(t_2^-)$ perpendicular to $\Om_j(V_2)$ or it collides with one of the boundaries of $\Om_j(V_2)$. Moreover $x|_{[t_1, t_2]}$ can always be extended to a collision-free or regularizable solution $x|_{\rr}$ of \eqref{eq;RNBP1} through the symmetries with respect to $V_1$ and $V_2$.  
\end{thm}

Notice that a solution of the $N$-body problem satisfying \textbf{Assumption \ref{asm;periodic}} must satisfy \textbf{Assumption \ref{asm;bi-hyp-para}} as well. In particular, every homographic periodic solution generated by a collinear Euler-Moulton configuration satisfies \textbf{Assumption \ref{asm;periodic}}. There are also many other periodic and quasi-periodic solutions in the $N$-body problem satisfying \textbf{Assumption \ref{asm;bi-hyp-para}} or \textbf{\ref{asm;periodic}} found through action minimization methods (see \cite{Ch08} and \cite{ChenLin09}). A list of examples will be given in Section \ref{sec;bounded}.




\section{A Marchal-type lemma for fixed-ends problems} \label{sec:local_minimizers_are_collision_free}


The Marchal-type lemma (Theorem~\ref{thm;Marchal-restricted}) for the restricted problem is highly nontrivial, 
as its argument about isolation of collision singularities differs significantly from the non-restricted version.
Before proceeding with the proof, we point out  some essential ingredients  of the proof for the $N$-body  problem (see \cite{BFT08,C02})
that are not directly applicable to the restricted problem. 

Following the approach for the $N$-body problem, the standard argument to show $x$ is collision-free is by contradiction.
Assume $x$ is not collision-free. The first step is to show it must contain an isolated collision. 
The second step is using Sundman's estimates to show $x$ cannot be a fixed-end local minimizer near an isolated collision. 

The proof of the first step relies on the Lagrange-Jacobi identity and the continuity of the partial energy of the system involved in collision. 
The continuity of partial energy is obtained through a time reparametrization of the critical path. 
While this works perfect for the time independent $N$-body problem, or even a time-dependent system in which case the singular set should be time-independent (see \cite{BFT08}), it does not apply to the restricted problem.

Here we pinpoint a few technical issues that result in failure of the method for the non-restricted problem:

\begin{enumerate}
\item For the restricted problem, the singular set is a time-varying subset in the configuration space. 
At each moment $t$, it consists of positions of primaries at time $t$.  
The primaries could be very close or far apart, we cannot assume the potential energy $U(x,t)$ in  (\ref{eq;Uzt}) 
tends to infinity uniformly in $t$ 
as the zero mass $x$ approaches the singular set.  \medskip

\item The energy of the system does not necessarily belong to the Sobolev space $W^{1,1}$ in some
small neighborhood of a collision moment. More precisely, let $x$ be an action minimizer and let 
$$
  h(t) \;=\; \frac{1}{2}|\dot{x}(t)|^2 - U(x(t),t)
$$
be its energy.  One may verify 
$$
  \dot{h}(t) \;=\; -\frac{\partial}{\partial t}U(x,t) \;=\;  -\sum_{i=1}^N \frac{m_i (x(t)-q_i(t))\cdot \dot{q}_i(t)}{|x(t)-q_i(t)|^3}
$$
in some weak sense, and it is not necessarily integrable. 
\medskip

\item  Isolation of collision moments for the $N$-body problem relies on positivity of the second derivative of the moment of inertia, 
and application of some Lagrange-Jacobi type inequality. 
If one follows the same idea for the restricted problem, the moment of inertia $I=|x|^2$ is expected to locally satisfy an inequality of the form
$$
   \frac{1}{2} \ddot{I} \; \ge \; 2 h(t) + (1-\epsilon) U(x,t) + C. 
$$
Here $\epsilon\ge 0$ is small and $C$ is some constant. 
As observed above, $h$ could be negative and unbounded, we can't conclude positivity of $\ddot{I}$ here even if $U$ tends to positive infinity. 
\end{enumerate}

In this section we demonstrate isolation of collision through a different argument, then use it to prove Theorem~\ref{thm;Marchal-restricted}. 
Throughout this section we only assume $q(t)$ is a collision-free solution of the $N$-body problem, and it does not need to be bounded. 


\begin{dfn}
\label{dfn;isolated-coll} We say $x(t_0)$ is \textbf{an isolated collision} and $t_0$ is \textbf{an isolated collision moment}, if there is a $\dl>0$ and $1 \le i_0 \le N$, such that 
$$ x(t_0) = q_{i_0}(t_0) \; \text{ and } \; x(t) \ne q_i(t), \; \forall 1 \le i \le N \text{ and } t \in [t_0-\dl, t_0+\dl] \setminus \{t_0\}. $$
In addition, if $x(t)$ satisfies \eqref{eq;RNBP1}, for all $t \in [t_0-\dl, t_0+\dl] \setminus \{t_0\}$, we say $x|_{[t_0 -\dl, t_0 +\dl]}$ is \textbf{a collision-ejection solution} of the restricted problem. 
\end{dfn}

The following result is known as \emph{Sundman's estimates}. 
\begin{prop}
    \label{prop;Sundman} If $x|_{[t_0-\dl, t_0+\dl]}$ is a collision-ejection solution of \eqref{eq;RNBP1} with $x(t_0) = q_{i_0}(t_0)$ being the isolated collision, there is an $\bs^+ \in \mathbb{S}$, such that when $t \to t_0^+$
   $$ \begin{aligned}
x(t) - q_{i_0}(t) & =  \left(\frac{9}{2}m_{i_0} \right)^{\frac{1}{3}}(t-t_0)^{\frac{2}{3}} \bs^+ + o\left(|t-t_0|^{\se} \right); \\
\dot{x}(t) - \dot{q}_{i_0}(t) & = \se \left(\frac{9}{2}m_{i_0} \right)^{\frac{1}{3}}(t-t_0)^{-\sy} \bs^+ + o\left(|t-t_0|^{-\frac{1}{3}} \right) .
\end{aligned}$$
There is an $\bs^- \in \mathbb S$, such that similar estimates hold when $t \to t_0^-$. 
\end{prop}

\begin{proof}
Let $\bfr=x-q_{i_0}$. Then the equation of motion for the massless particle is 
\begin{equation}
\label{eq;r} \ddot{\bfr} = \ddot{x} - \ddot{q}_{i_0} = - \frac{m_{i_0}\, \bfr}{|\bfr|^3} - \ddot{q}_{i_0}  - \sum_{i \ne i_0} \frac{m_i(\bfr + q_{i_0} -q_i)}{|\bfr + q_{i_0} -q_i|^3}.
\end{equation}
Notice that there is a constant $C_1$, such that for all $t \in [t_0 -\dl, t_0+\dl]$, 
$$ \Big|\ddot{q}_{i_0}(t)+\sum_{i \ne i_0} \frac{m_i(\bfr(t) +q_{i_0}(t) - q_i(t))}{|\bfr(t) +q_{i_0}(t) - q_i(t)|^3} \Big| \le C_1.$$
This means $\bfr(t)$ is a collision-ejection solution of the Kepler problem with bounded perturbation. Then the desired results follow from (7.11) and (7.12) in \cite{Sp69} by Sperling. 

\end{proof}

The next result is the key to obtain an isolated collision in our proof of Theorem \ref{thm;Marchal-restricted}.
\begin{prop} \label{prop;isolated-coll}
If $x \in H^1_{t_0-\dl, t_0+\dl}(\rr^d)$ is a fixed-end local minimizer of $\A$ with $t_0$ being a collision moment, then $t_0$ is either isolated or accumulated by collision moments from both sides.  
\end{prop}

\begin{proof}
Without loss of generality, let's assume $x(t_0) = q_{i_0}(t_0)$ and $\dl$ is small enough such that there are two positive constants $C_i$, $i =1, 2$, satisfying
\begin{equation}
\label{eq;C1} \max \{ |\ddot{q}_{i_0}(t)|: \forall t \in [t_0 -\dl, t_0 +\dl] \} \le C_1;
\end{equation}
\begin{equation}
\label{eq;C2} \min \{ |x(t) - q_i(t)|: \; \forall t \in [t_0 -\dl, t_0+\dl] \; \text{and } \forall i \ne i_0 \} \ge C_2^{-1}. 
\end{equation}
Then the massless body can only collide with $m_{i_0}$. Denote the set of collision moments by
\begin{equation*}
\label{eq;Z} Z(t_0,\delta)\;=\; \{t\in (t_0-\delta,t_0+\delta): \; x(t) - q_{i_0}(t)=0\}.
\end{equation*}

By a contradiction argument, let's assume $ Z(t_0,\delta)$ is infinite and $t_0\in  Z(t_0,\delta)$ is accumulated from one side. 
On $Z(t_0,\delta)$ the Lagrangian is infinite, so $Z(t_0,\delta)$ has zero Lebesgue measure. 
It is closed, so its complement can be expressed as a countable union of disjoint open intervals. 

There are two cases: (i) $t_0$ is accumulated from the right but not from the left;
(ii) $t_0$ is accumulated from the left but not from the right.

Consider case (i). 
As $t_0$ is accumulated from the right, we can find open intervals $\{(a_k,b_k)\}$ in 
$(t_0,t_0+\delta)\setminus Z(t_0,\delta)$ such that $t_0<a_{k+1} < b_{k+1} < a_k < b_k$ for each $k\in \N$,  
and both $\{a_k\}$, $\{b_k\}$ are sequences of collision moments and they both converge  to $t_0$. Then
$$ \dl_k = b_k -t_0 \to 0, \;\text{ as } k \to \infty. $$
Since $t_0$ is not accumulated from the left, when $k$ is large enough, there is no collision on $(t_0-\delta_k, t_0)$. 

Let $\bfr=x-q_{i_0}$ and $m_{i_0} = \mu$. Then 
$$ \A(x; t_0-\dl, t_0 +\dl) = \int_{t_0 -\dl}^{t_0 +\dl} \frac{1}{2}|\dot{\bfr}|^2 + \frac{\mu}{|\bfr|} + \sum_{i \ne i_0} \frac{m_i}{|\bfr + q_{i_0} - q_i|} + \ey |\dot{q}_{i_0}|^2 + <\dot{\bfr}, \dot{q}_{i_0} > \,dt. $$
Meanwhile we define a new path $\tilde{x} = \tilde{\bfr}+ q_{i_0} \in H^1_{t_0 -\dl, t_0 +\dl}(\rr^d)$ with
$$
 \tilde{\bfr} \;=\; \left\{ \begin{array}{ll} \bfr(t-\dl_k) & \text{if} \;\; t\in (t_0,t_0+\dl_k], \\ 
 \bfr(t_0-\dl_k) & \text{if} \;\; t\in (t_0-\dl_k, t_0], \\
 \bfr(t)   &   \text{otherwise.}
 \end{array}\right. 
$$
Its action value is
$$ \A(\tilde{x}; t_0-\dl, t_0 +\dl) = \int_{t_0 -\dl}^{t_0 +\dl} \frac{1}{2}|\dot{\tilde{\bfr}}|^2 + \frac{\mu}{|\tilde{\bfr}|} + \sum_{i \ne i_0} \frac{m_i}{|\tilde{\bfr} + q_{i_0} - q_i|} + \ey |\dot{q}_{i_0}|^2 + <\dot{\tilde{\bfr}}, \dot{q}_{i_0} > \,dt. $$
We get a contradiction, when the following holds  
\begin{equation*}
\label{eq;dA} \A(x; t_0-\dl, t_0 +\dl) -\A(\tilde{x}; t_0-\dl, t_0 +\dl) >0
\end{equation*}

To show this, first we notice that the continuity of $x(t)$ implies 
\begin{equation}
\label{eq;ep-k} \ep_k = \max \{|\bfr (t)|: \; \forall t \in [t_0 -\dl_k, t_0 +\dl_k] \} \to 0, \; \text{ as } k \to \infty. 
\end{equation}
Meanwhile by Sundman's estimates (see Proposition \ref{prop;Sundman}), for $k$ large enough
\begin{eqnarray}
  \frac{1}{|\bfr (t)|} &\le & 2 \left(\frac{2}{9\mu} \right)^\frac{1}{3} |t-t_0|^{-\frac{2}{3} }, 
   \quad  \forall t\in [t_0-\delta_k,t_0).  \label{eqn:Sund1}
\end{eqnarray}
and by the definition of $\tilde{\bfr}$,  
\begin{equation}
\label{eq;tl-r-le} |\tilde{\bfr}(t)| \le  2 \left(\frac{9\mu}{2} \right)^\frac{1}{3} \dl_k^{\frac{2}{3} }, \;\; \forall t \in [t_0 -\dl_k, t_0 +\dl_k]. 
\end{equation}

Now note that $\bfr(t_0)=\bfr(t_0+\dl_k)=0$. Then by Gordon's theorem~\cite{Gor77} and \eqref{eqn:Sund1}. 
\begin{equation}
\label{eq;dA-1}  \begin{aligned}
\int_{t_0 -\dl_k}^{t_0+\dl_k} \ey|\dot{\bfr}|^2 + \frac{\mu}{|\bfr|} \,dt & - \int_{t_0-\dl_k}^{t_0 +\dl_k} \ey |\dot{\tilde{\bfr}}|^2 + \frac{\mu}{|\tilde{\bfr}|} \,dt \\
& = \int_{t_0}^{t_0+\dl_k} \ey|\dot{\bfr}|^2 + \frac{\mu}{|\bfr|} \,dt -\int_{t_0-\dl_k}^{t_0}\ey |\dot{\tilde{\bfr}}|^2 + \frac{\mu}{|\tilde{\bfr}|} \,dt \\
& \ge 3 \left(\frac{\mu^2 \pi^2}{2}\right)^\frac{1}{3} \dl_k^\frac{1}{3} - \frac{\mu\,  \dl_k}{|\bfr(t_0-\dl_k)|} \\
& \ge \left( 3 \left(\frac{\pi^2}{2}\right)^{\frac{1}{3}} - 2\left( \frac{2}{9}\right)^{\frac{1}{3}} \right) \mu^{\se} \dl_k^{\frac{1}{3}}.
\end{aligned}
\end{equation}

Meanwhile since $x \in H^1_{t_0 -\dl, t_0+\dl}$ and $q_{i_0} \in \mathcal{C}^1$, $\bfr \in H^1_{t_0-\dl, t_0+\dl}$ as well. Then by the definition of $\tilde{\bfr}$, it also belongs to $H^1_{t_0 -\dl, t_0+\dl}$. As a result, using integration by parts  
\begin{equation} \label{eq;dA-3}
\begin{aligned}
\left| \int_{t_0 -\dl_k}^{t_0+\dl_k} < \dot{\bfr}- \dot{\tilde{\bfr}}, \dot{q}_{i_0}> \,dt \right| & = \left| \int_{t_0 -\dl_k}^{t_0 +\dl_k} <\bfr -\tilde{\bfr}, \ddot{q}_{i_0}> \,dt \right|  \le  \int_{t_0 -\dl_k}^{t_0 +\dl_k} |\ddot{q}_{i_0}| (|\bfr| + |\tilde{\bfr}|) \,dt     \\
& \le 2 C_1 \dl_k \left(\ep_k + 2 \left(\frac{9 \mu}{2} \right)^{\frac{1}{3}} \dl_k^{\se} \right). 
\end{aligned}
\end{equation}
where the last inequality follows from \eqref{eq;C1}, \eqref{eq;ep-k} and \eqref{eq;tl-r-le}. 

Finally recall that $m= \sum_{i}m_i$, for all $ t \in [t_0 -\dl, t_0 +\dl]$, by \eqref{eq;C2}
$$ \sum_{i \ne i_0} \frac{m_i}{|\bfr + q_{i_0} -q_i|} = \sum_{i \ne i_0} \frac{m_i}{|x -q_i|} \le (m -m_{i_0}) C_2, $$ 
and by \eqref{eq;ep-k} and \eqref{eq;tl-r-le}, for $k$ large enough
$$ \begin{aligned}
\sum_{i \ne i_{0}} \frac{m_i}{|\tilde{\bfr} + q_{i_0} - q_i|} & = \sum_{i \ne i_{0}} \frac{m_i}{|x-q_i + \tilde{\bfr} -\bfr|} \le  \sum_{i \ne i_{0}}\frac{m_i}{ |x-q_i| -| \tilde{\bfr} -\bfr|} \\
& \le   \sum_{i \ne i_{0}} m_i \left( C_2^{-1} - 2 \left(\frac{9\mu}{2} \right)^\frac{1}{3} \dl_k^{\frac{2}{3} } - \ep_k \right)^{-1} \\
& \le 2 (m -m_{i_0})C_2
\end{aligned}
$$
The estimates imply 
\begin{equation}
\label{eq;dA-2}  \left| \int_{t_0 -\dl_k}^{t_0 +\dl_k}\sum_{i \ne i_0} \frac{m_i}{|\bfr + q_{i_0} - q_i|}-  \sum_{i \ne i_0} \frac{m_i}{|\tilde{\bfr} + q_{i_0} - q_i|}  \,dt\right|  \le 3m C_2\dl_k. 
\end{equation}

According to \eqref{eq;dA-1}, \eqref{eq;dA-2} and \eqref{eq;dA-3}
$$ \begin{aligned}
\A(x; & t_0-\dl, t_0 +\dl) -\A(\tilde{x}; t_0-\dl, t_0 +\dl) \\
& \ge  \left( 3 \left(\frac{\pi^2}{2}\right)^{\frac{1}{3}} - 2\left( \frac{2}{9}\right)^{\frac{1}{3}} \right) \mu^{\se} \dl_k^{\frac{1}{3}} - (2 C_1 \ep_k + 3 m C_2) \dl_k - 4 C_1 \left( \frac{9 \mu}{2} \right)^{\frac{1}{3}} \dl_k^{\frac{5}{3}} \\
& >0, \; \text{ when $k$ is large enough}.  
\end{aligned}
$$
\end{proof}

\begin{prop}
\label{prop;coll-free} If $x \in H^1_{t_0-\dl, t_0+\dl}(\rr^d)$ is a collision-ejection solution of $\eqref{eq;RNBP1}$ with an isolated collision $x(t_0)$, then it can not be a fixed-ends local minimizer of $\A$. 
\end{prop}
This follows from a blow-up argument and the details will be deferred to the next section. 

\begin{proof}[Proof of Theorem \ref{thm;Marchal-restricted}]
By a contradiction argument let us assume the local minimizer $x \in H^1_{t_1, t_2}(\rr^d)$ is not collision-free, and denote the set of collision moments by 
$$ Z(t_1, t_2) = \{ t \in (t_1, t_2): x(t) = q_{i}(t), \text{ for some } i \}. $$
Since $\A(x; t_1, t_2)$ is finite, $Z(t_1, t_2)$ is a subset of $(t_1, t_2)$ with zero Lebesgue measure. 
Then there must exist an open interval $(\tau_1, \tau_2) \subset (t_1, t_2)$, such that $x(t)$ is collision-free for all $t \in (\tau_1, \tau_2)$, and 
$$ t_1 \ne \tau_1 \in Z(t_1, t_2) \text{ or } \; t_2 \ne \tau_2 \in Z(t_1, t_2). $$

Without loss of generality, let's assume the former case holds. Then $\tau_1$ is a collision moment, and it 
can not be accumulated by collision moments from both sides. By Proposition \ref{prop;isolated-coll}, $x(\tau_1)$ 
must be an isolated collision, and then by Proposition \ref{prop;coll-free}, $x$ can not be a fixed-end local minimizer, which is a contradiction to our assumption.  
\end{proof}


\section{Proof of Proposition \ref{prop;coll-free}} \label{sec;blow-up}
To simplify notation, there is no loss of generality to assume the isolated collision moment $t_0=0$ with $x(0)= q_N(0)$. We further set 
$$ \bfr = x - q_N, \text{ and } \mu = m_N. $$
Then $\A(x; -\dl, \dl)$ can be rewritten as 
\begin{equation*}
\A(\bfr;  -\dl,  \dl)  = \int_{t_0 -\dl}^{t_0 +\dl} \frac{1}{2}|\dot{\bfr} + \dot{q}_N|^2 + \frac{\mu}{|\bfr|} + \sum_{i=1}^{N-1} \frac{m_i}{|\bfr + q_{N} - q_i|} \,dt.
\end{equation*}

Our proof is based on a blow-up argument from \cite{FT04} and a certain variational property of the Kepler problem. Consider the following Kepler problem in $\rr^d$
\begin{equation}
\label{eq;kepler} \ddot{\xi} = -\mu \frac{\xi}{|\xi |^3}, \;\; \mu >0.
\end{equation}
It is the Euler-Lagrange equation of the Lagrangian action functional $\bar{\A}$
\begin{equation}
\label{eq;LagKepler} \bar{\A}(\xi; t_1, t_2) = \int_{t_1}^{t_2} \bar{L}(\xi, \dot{\xi}) \,dt, \; \text{ where } \bar{L}(\xi, \dot{\xi}) = \ey |\dot{\xi}|^2 + \frac{\mu}{|\xi|}. 
\end{equation}

\begin{prop}
\label{prop;rbar-not-minimizer} For any $\bs^{\pm} \in \mathbb S$, let $\bar{\bfr}$ be \emph{a parabolic collision-ejection solution} of \eqref{eq;kepler} with
\begin{equation}
\label{eq;Kepler-Coll-Eject} \bar{\bfr}(t) = \begin{cases}
\left(\frac{9}{2}\mu \right)^{\frac{1}{3}} |t|^{\frac{2}{3}} \bs^-, & \text{ if } t \le 0; \\
\left(\frac{9}{2}\mu \right)^{\frac{1}{3}} |t|^{\frac{2}{3}} \bs^+, & \text{ if } t \ge 0.
\end{cases}
\end{equation} 
For any $T>0$, there is a collision-free solution $\rt(t)$ of \eqref{eq;kepler} satisfying 
$$  \bar{\A}(\rt; -T, T)  < \bar{\A}(\bar{\bfr}; -T, T) \; \text{ and } \; \rt(\pm T) = \bar{\bfr}(\pm T). $$
\end{prop}
This is a well-known result and the first proofs can be found in \cite{Mc02} and \cite{C02}.

Recall that by Sundman's estimates (Proposition \ref{prop;Sundman}), there exists $\bs^{\pm} \in \mathbb S$, such that
$$ \lim_{t \to 0^{\pm}} \frac{\bfr(t)}{|\bfr(t)|} =  \lim_{t \to 0^{\pm}} \frac{x(t)-q_N(t)}{|x(t)-q_N(t)|} = \bs^{\pm}. $$
From now on let $\bar{\bfr}$ be the parabolic collision-ejection solution as defined in \eqref{eq;Kepler-Coll-Eject} with the above $\bs^{\pm}$. 
\begin{dfn}
\label{def;blow-up}
For any $\lmd>0$, define the \textbf{$\lmd$-blow-up} of $\bfr$ and $q_i$, $i =1, \dots, N$, as
\begin{equation}
\label{eq;blowup-r} \bfr^{\lmd}(s) = \lmd^{\se} \bfr(\lmd^{-1}s) \; \text{ and } \;  \ql_i(s) = \lmd^{\se} q_i(\lmd^{-1} s), \;\; \forall s \in [-\lmd \dl, \lmd \dl]. 
\end{equation}
\end{dfn}
Notice that $\bfr^{\lmd}|_{(-\lmd \dl, 0) \cup (0, \lmd \dl)}$ is a solution of 
\begin{equation}
\label{eq;ResNbody-blowup} \ddot{\bfr}^{\lmd}  + \ddot{q}_i^{\lmd}  = - \mu \frac{\bfr^{\lmd} }{|\bfr^{\lmd} |^3} - \sum_{i=1}^{N-1} \frac{m_i(\bfr^{\lmd}  + \ql_{N}  - \ql_i )}{|\bfr^{\lmd}  + \ql_{N}  - \ql_i |^3} 
\end{equation}
which is the Euler-Lagrangian equation of the following action functional
\begin{equation}
\A^{\lmd}(\bfr^{\lmd}; s_1, s_2) = \int_{s_1}^{s_2} L^{\lmd}(\bfr^{\lmd} , \dot{\bfr}^{\lmd} , s) \,ds
\end{equation}
where
\begin{equation}
\label{eq;LagRes-blowup} L^{\lmd}(\bfr^{\lmd} , \dot{\bfr}^{\lmd} , s) = \ey |\dot{\bfr}^{\lmd}  + \dot{q}^{\lmd}_{N}  |^2 + \frac{\mu}{|\bfr^{\lmd} |} + \sum_{i =1}^{N-1} \frac{m_i}{|\bfr^{\lmd}  + \ql_{N}  - \ql_i |}.
\end{equation}

\begin{lem}
\label{lem:ScaleofAction} For any $\dl_1 \in (0, \dl]$,  $\A^{\lmd}(\bfr^{\lmd}; -\lmd \dl_1, \lmd \dl_1) = \lmd^{\frac{1}{3}} \A (\bfr; -\dl_1, \dl_1). $ 
\end{lem}

\begin{proof}
By a direct computation, 
$$ L^{\lmd}(\bfr^{\lmd}(s), \dot{\bfr}^{\lmd}(s), s) = \lmd^{-\se} L(\bfr(t), \dot{\bfr}(t), t). $$
The result then follows from the fact that $t = \lmd^{-1}s$, as it implies $ds = \lmd dt$. 
\end{proof}
 
\begin{lem}
\label{lem:ztl-to-ztb} For any $T>0$, as $\lmd$ goes to infinity, 
\begin{enumerate}
\item[(a).] $\bfr^{\lmd}(s)$ converges uniformly to $\bar{\bfr}(s)$ on $[-T, T]$;
\item[(b).] $\dot{\bfr}^{\lmd}(s)$ converges uniformly to $\dot{\bar{\bfr}}(s)$ on any compact subset of $[-T, T] \setminus \{0\}$;
\end{enumerate}
\end{lem}

\begin{proof}
It follows from Proposition \ref{prop;Sundman} and direct computations. 
\end{proof}

From now on let's fix a $T>0$. By Proposition \ref{prop;rbar-not-minimizer}, we may assume
\begin{equation}
\label{eq;ep0} \bar{\A}(\bar{\bfr}; -T, T) - \bar{\A}(\rt; -T, T) = 6\ep_0 >0. 
\end{equation}

Meanwhile by Lemma \ref{lem:ztl-to-ztb}, $\{\bfr^{\lmd}|_{[-2T, 2T]}: \lmd >2T/\dl \}$ is a bounded subset in $H^1_{-2T, 2T}(\rr^d)$. Then we can find a sequence $\{ \lmd_n \nearrow \infty\}_{n=1}^{\infty}$, such that $\bfr^{\lmd_n}|_{[-2T, 2T]}$ converges weakly to $\bar{\bfr}|_{[-2T, 2T]}$, as $n$ goes to infinity. The weakly lower semi-continuity of $\bar{\A}$ then implies   
\begin{equation} \label{eq;Abar-Ztlmd-Ztbar}
\bar{\A}(\bfr^{\lmd_n}; -2T, 2T) \ge \bar{\A}(\bar{\bfr}; -2T, 2T)-\ep_0, \; \text{ for $n$ large enough}.
\end{equation}

For each $n$ large enough, define a new path $\bfu^{\lmd_n} \in H^1_{-\lmd_n \dl, \lmd_n \dl}(\rr^d)$ as 
\begin{equation}
\label{eq;XiLmdn} \bfu^{\lmd_n}(s) = \begin{cases}
\bfr^{\lmd_n}(s), \; & \text{ when } s \in [2T, \lmd_n\dl], \\
\frac{2T-s}{\dl_n} (\bar{\bfr}(s) - \bfr^{\lmd_n}(s)) + \bfr^{\lmd_n}(s), \; & \text{ when } s \in [2T- \dl_n, 2T], \\
\bar{\bfr}(s), \; & \text{ when } s \in [T, 2T- \dl_n], \\
\rt(s), \; & \text{ when } s \in [-T, T], \\
\bar{\bfr}(s), \; & \text{ when } s \in [-2T+ \dl_n, -T], \\
\frac{s + 2T-\dl_n}{\dl_n}(\bar{\bfr}(s) - \bfr^{\lmd_n}(s)) + \bar{\bfr}(s), \; & \text{ when } s \in [-2T, -2T + \dl_n], \\
\bfr^{\lmd_n}(s), \; & \text{ when } s \in [-\lmd_n\dl, -2T],
\end{cases} 
\end{equation}
where $\dl_n >0$ is a proper sequence of small enough numbers. 

We claim the following inequality holds, whose proof will be postponed for a moment 
\begin{equation}
\label{eq;AlmdXi<Zt} \A^{\lmd_n}(\bfu^{\lmd_n}; -\lmd_n\dl, \lmd_n\dl) < \A^{\lmd_n}(\bfr^{\lmd_n}; -\lmd_n\dl, \lmd_n\dl)\;\; \text{ for } n \text{ large enough}.
\end{equation}
For each $n$ large enough, we define a new path $\bfv_n \in H^1_{-\dl, \dl}(\rr^d)$ as below, 
$$ \bfv_n(t) = \lmd_n^{-\se} \bfu^{\lmd_n}(\lmd_n t), \; \forall t \in [-\dl, \dl]. $$
Since the $\lmd_n$-blow-up of $\bfv_n$ is nothing but $\bfu^{\lmd_n}$, by \eqref{eq;XiLmdn}
$$ \bfv_n(\pm \dl) + q_N(\pm \dl)= \lmd_n^{-\se} \bfr^{\lmd_n}(\pm \lmd_n \dl) + q_N(\pm \dl)=\bfr(\pm \dl) + q_N(\pm \dl) = x(\pm \dl) . $$
Then Proposition \ref{prop;coll-free} follows from the estimates below, which are implied by Lemma \ref{lem:ScaleofAction} and the above claim
$$ \begin{aligned}
\A(x; -\dl, \dl) & - \A(\bfv_n + q_N; -\dl, \dl) = \A(\bfr; -\dl, \dl)  - \A(\bfv_n; -\dl, \dl) \\
& = \lmd_n^{-\sy} \left( \A^{\lmd_n}(\bfr^{\lmd_n}; -\lmd_n\dl, \lmd_n\dl) - \A^{\lmd_n}(\bfu^{\lmd_n}; -\lmd_n\dl, \lmd_n\dl) \right)>0.
\end{aligned} $$

To prove the claim \eqref{eq;AlmdXi<Zt}, notice that by the definition of $\bfu^{\lmd_n}$
\begin{equation} \label{eq;Act-Diff-ztlmn-xilmn}
 \begin{aligned}
 \A^{\lmd_n}(\bfr^{\lmd_n}; & -\lmd_n\dl, \lmd_n\dl) - \A^{\lmd_n}(\bfu^{\lmd_n}; -\lmd_n\dl, \lmd_n\dl)   \\
 & =  \A^{\lmd_n}(\bfr^{\lmd_n}; -2T, 2T) - \A^{\lmd_n}(\bfu^{\lmd_n}; -2T, 2T) \\
 & \ge \left( \bar{\A}(\bfr^{\lmd_n}; -2T, 2T) - \bar{\A}(\bfu^{\lmd_n}; -2T, 2T) \right) + \int_{-2T}^{2T} \langle \dot{\bfr}^{\lmd_n} - \dot{\bfu}^{\lmd_n}, \dot{q}_{N}^{\lmd_n} \rangle \,ds \\
 & - \int_{-2T}^{2T} \sum_{i =1}^{N-1} \frac{m_i}{|\bfu^{\lmd_n} + q^{\lmd_n}_{N} - q^{\lmd_n}_i|} \,ds.
 \end{aligned}
\end{equation}
Hence we shall divide the proof into three steps accordingly. 

\emph{Step 1}: 
By Lemma \ref{lem:ztl-to-ztb}, there is a sequence of positive numbers $\{\dl_n \searrow 0 \}_{n =1}^{\infty}$, such that 
\begin{equation}
\label{eq;LeDeltan} |\bfr^{\lmd_n}(s) - \bar{\bfr}(s)| \le \dl_n, \;\; |\dot{\bfr}^{\lmd_n}(s) - \dot{\bar{\bfr}}(s)| \le \dl_n, \;\; \forall s \in [-2T, 2T] \setminus [-T, T]. 
\end{equation}
Combining this with \eqref{eq;XiLmdn}, for $n$ large enough, we can find two positive constants $C_1$ and $C_2$ independent of $n$, such that 
\begin{equation} \label{eq;xilmd-upper-bound}
|\dot{\bfu}^{\lmd_n}(s)|\le C_1, \;  \forall s \in [-2T, -2T + \dl_n] \cup [2T- \dl_n, 2T].
\end{equation} 
\begin{equation} \label{eq;xilmd-upper-bound-2}
  |\bfu^{\lmd_n}(s)| \ge C_2^{-1}, \;  \forall s \in [-2T, -2T + \dl_n] \cup [2T- \dl_n, 2T]. 
 \end{equation}  
Then for $n$ large enough, 
$$
 \begin{aligned}
 \bar{\A}(\bfu^{\lmd_n}; & -2T, -2T+\dl_n)  + \bar{\A}(\bfu^{\lmd_n}; 2T-\dl_n, 2T) \\
 & \le  \int_{-2T}^{-2T + \dl_n} + \int_{2T-\dl_n}^{2T} \ey |\dot{\bfu}^{\lmd_n}(s)|^2 + \frac{\mu}{|\bfu^{\lmd_n}(s)|} \,ds \le 2\left( \ey C_1^2 + \mu C_2 \right)\dl_n \le \ep_0.  
 \end{aligned}
 $$
Together with \eqref{eq;ep0} and \eqref{eq;XiLmdn}, they imply
\begin{equation*}
\label{eq;Abar-XiLmd-Ztbar}  
\begin{aligned}
\bar{\A}(\bar{\bfr}; -2T, 2T) & - \bar{\A}(\bfu^{\lmd_n}; -2T, 2T) \\
&  \ge  6\ep_0 -  \bar{\A}(\bfu^{\lmd_n}; -2T, -2T+\dl_n)- \bar{\A}(\bfu^{\lmd_n}; 2T-\dl_n, 2T) \ge 5 \ep_0.
\end{aligned}
\end{equation*}

Then \eqref{eq;Abar-Ztlmd-Ztbar} implies 
\begin{equation}
\label{eq;Almd-1} \bar{\A}(\bfr^{\lmd_n}; -2T, 2T) - \bar{\A}(\bfu^{\lmd_n}; -2T, 2T) \ge 4 \ep_0. 
\end{equation}

\emph{Step 2}: By Sundman's estimates,  
\begin{equation}
\label{eq;Zt-Dot-Upper-Bound} |\dot{\bfr}(t)| \le \left( \frac{9 \mu}{2} \right)^{\frac{1}{3}} |t|^{-\sy}, \text{ when $|t|$ is small enough.}  
\end{equation}
Meanwhile as $q_N \in \mathcal{C}^1$, 
$$ \max\{ |\dot{q}_N(t)|: t \in [-\dl, \dl] \} = C_3 < \infty. $$
Now by the definition of $\lmd_n$-blow up, 
$$ \begin{aligned}
\int_{-2T}^{2T}  |\dot{\bfr}^{\lmd_n}(s)| \cdot |\dot{q}_{N}^{\lmd_n}(s)| \,ds & = \lmd_n^{\sy} \int_{- \frac{2T}{\lmd_n}}^{ \frac{2T}{\lmd_n}} |\dot{\bfr}(t)| \cdot |\dot{q}_{N}(t)| \,dt \\
 & \le \lmd_n^{\sy} C_3 \left( \frac{9 \mu}{2} \right)^{\frac{1}{3}} \int_{- \frac{2T}{\lmd_n}}^{\frac{2T}{\lmd_n}} |t|^{-\sy}\,dt  \\
& = 3(2T)^{\se} C_3 \lmd_n^{-\sy} \left( \frac{9 \mu}{2} \right)^{\frac{1}{3}} \le \ep_0
\end{aligned}$$
for $n$ large enough. 

Notice that $\forall s \in [-2T, 2T]$, 
$$|\dot{q}^{\lmd_n}_{N}(s)| = \lmd_n^{-\sy} |\dot{q}_{N}(\lmd_n^{-1} s)| \le \lmd_n^{-\sy}C_3$$
Combining this with \eqref{eq;XiLmdn} and \eqref{eq;xilmd-upper-bound}, for $n$ large enough 
$$  \int_{-2T}^{2T}  |\dot{\bfu}^{\lmd_n}(s)| \cdot |\dot{q}_{N}^{\lmd_n}(s)| \,ds \le \lmd_n^{-\sy}C_3 \int_{-2T}^{2T}  |\dot{\bfu}^{\lmd_n}(s)| \,ds \le \lmd_n^{-\sy}C_4 \le \ep_0.$$

As a result, for $n$ large enough
\begin{equation}
\label{eq;Almd-2} \int_{-2T}^{2T}  |\dot{\bfr}^{\lmd_n}(s)| \cdot |\dot{q}_{N}^{\lmd_n}(s)| \,ds + \int_{-2T}^{2T}  |\dot{\bfu}^{\lmd_n}(s)| \cdot |\dot{q}_{N}^{\lmd_n}(s)| \,ds \le 2\ep_0.
\end{equation}

\emph{Step 3}: Recall that 
$$ \rho=\min \{|x(t) - q_i(t)|: \forall t \in [-\dl,  \dl] \text{ and } i \ne N \} >0. $$
Meanwhile by \eqref{eq;XiLmdn}, for any $n$ large enough,  
\begin{equation*}
\label{eq;eta-zt-ep}  |\bfv_n(t) - \bfr(t)| = |\bfv_n(t) +q_{N}(t) - x(t)|\le \rho/2, \;\; \forall t \in [-\dl, \dl]. 
\end{equation*}
This implies that for all $t \in [-\dl, \dl]$
$$ |\bfv_n(t) + q_{N}(t) - q_i(t)|  \ge |\bfr(t) +  q_{N}(t) - q_i(t)| - |\bfv_n(t) - \bfr(t)| \ge \rho/2.$$
As a result, for $n$ large enough, 
\begin{equation}
\label{eq;Almd-3} \begin{aligned}
\int_{-2T}^{2T}   \sum_{i =1}^{N-1} & \frac{m_i}{|\bfu^{\lmd_n}(s) + q^{\lmd_n}_{N}(s) - q^{\lmd_n}_i(s)|} \,ds \\
& = \lmd_n^{\sy} \int_{ -\frac{2T}{\lmd_n}}^{\frac{2T}{\lmd_n}} \sum_{i =1}^{N-1}  \frac{m_i}{|\bfv_n(t) + q_{N}(t) - q_i(t)|} \,dt \\
& \le \lmd_n^{\sy} \int_{ -\frac{2T}{\lmd_n}}^{+\frac{2T}{\lmd_n}} \frac{2(m -\mu)}{\rho } \,dt \\
& \le \frac{8(m-\mu)T}{\rho} \lmd_n^{-\se} \le \ep_0. 
\end{aligned}
\end{equation}

The claimed inequality \eqref{eq;AlmdXi<Zt}  now follows from \eqref{eq;Almd-1}, \eqref{eq;Almd-2} and \eqref{eq;Almd-3}. 
 

\section{A Marchal-type lemma for free-boundary problems}

This section is devoted to a proof of Theorem \ref{thm;Marchal-rest-2} and the approach again is a combination of the blow-up argument and a variational property of the Kepler problem given as below. 

\begin{prop}
\label{prop;rbar-not-minimizer-2} Let $\bar{\bfr}(s)$ be a parabolic collision-ejection solution of the Kepler problem as given in \eqref{eq;Kepler-Coll-Eject}, then for any  $\bs^{\pm} \in \mathbb S$ and $T>0$, 
\begin{enumerate}
    \item[(a).] if $ \bs \ne - \bs^+ $, there is a collision-free solution $\rt(t)$ of \eqref{eq;kepler} satisfying 
    $$ \rt(0)/|\rt(0)| = \bs, \; \rt(T) = \bar{\bfr}(T) \text{ and } \bar{\A}(\rt; 0, T) < \bar{\A}(\bar{\bfr}; 0, T). $$
    \item[(b).] if $ \bs \ne - \bs^- $, there is a collision-free solution $\rt(t)$ of \eqref{eq;kepler} satisfying 
        $$ \rt(0)/|\rt(0)| = \bs, \; \rt(-T) = \bar{\bfr}(-T) \text{ and } \bar{\A}(\rt; -T, 0) < \bar{\A}(\bar{\bfr}; -T, 0). $$

\end{enumerate}
\end{prop}
This is again a well-known result and different proofs can be found in \cite{TV07}, \cite{FGN11}, \cite{Y15a} and \cite{Chen17}.

\begin{proof}[Proof of Theorem \ref{thm;Marchal-rest-2}]

The fact that $x(t)$ is collision-free, for all $t \in (t_1, t_2)$ follows directly from Theorem \ref{thm;Marchal-restricted}. 

(a). We will give the details for $X_1$ (the proof for $X_2$ is similar). By a contradiction argument, assume $x(t_1)$ is a collision. By what we have already proved it must be isolated. To simplify notation, there is no loss of generality to assume $t_1 =0$ and $x(0) = q_N(0)$, and following a similar setting from Section \ref{sec;blow-up}, let
$$ \bfr = x - q_N, \text{ and } \mu = m_N. $$
Again by Sundman's estimates (Proposition \ref{prop;Sundman}), there exists $\bs^{+} \in \mathbb S$, such that
$$ \lim_{t \to 0^{+}} \frac{\bfr(t)}{|\bfr(t)|} =  \lim_{t \to 0^{+}} \frac{x(t)-q_N(t)}{|x(t)-q_N(t)|} = \bs^{+}. $$
Let $\bar{\bfr}(t)$, $t \in [0, \infty)$, be a parabolic ejection solution of the Kepler problem. 

Fix a $\dl>0$ small enough, following the blow-up argument in Section \ref{sec;blow-up}, for any $\lmd >0$, we set
 $$ \bfr^{\lmd}(s) = \lmd^{\se} \bfr(\lmd^{-1}s) \; \text{ and } \;  \ql_i(s) = \lmd^{\se} q_i(\lmd^{-1} s), \;\; \forall s \in [0, \lmd \dl].
 $$
Then the following results can be proven by the same arguments used in Lemma \ref{lem:ScaleofAction} and \ref{lem:ztl-to-ztb}.
\begin{lem}
\label{lem:ScaleofAction-2} For any $\dl_1 \in (0, \dl]$,  $\A^{\lmd}(\bfr^{\lmd}; 0, \lmd \dl_1) = \lmd^{\frac{1}{3}} \A (\bfr; 0, \dl_1). $ 
\end{lem}
 
\begin{lem}
\label{lem:ztl-to-ztb-2} For any $T>0$, as $\lmd$ goes to infinity, $\bfr^{\lmd}(s)$ converges uniformly to $\bar{\bfr}(s)$ on $[0, T]$, and $\dot{\bfr}^{\lmd}(s)$ converges uniformly to $\dot{\bar{\bfr}}(s)$ on any compact subset of $(0, T]$.

\end{lem}

By condition $(1)$ in Definition \ref{dfn;free-boundary}, $X_1 \cap \mathbb S$ is either a ${\rm dim}(X_1)-1$ dimensional sphere or a set with two different points. In either case, we can always find a $\bs \in X_1 \cap \mathbb S$, such that $\bs \ne -\bs^+$.

Then for any fixed $T>0$, we can find a collision-free solution $\tilde{\bfr}|_{[0, T]}$ of the Kepler problem satisfying property (a) in Proposition \ref{prop;rbar-not-minimizer-2}. Now for each $n$ large enough, we define a new path 
\begin{equation}
\label{eq;XiLmdn-2} \bfu^{\lmd_n}(s) = \begin{cases}
\bfr^{\lmd_n}(s), \; & \text{ when } s \in [2T, \lmd_n\dl], \\
\frac{2T-s}{\dl_n} (\bar{\bfr}(s) - \bfr^{\lmd_n}(s)) + \bfr^{\lmd_n}(s), \; & \text{ when } s \in [2T- \dl_n, 2T], \\
\bar{\bfr}(s), \; & \text{ when } s \in [T, 2T- \dl_n], \\
\rt(s), \; & \text{ when } s \in [0, T].
\end{cases} 
\end{equation}
By a similar argument as in Section \ref{sec;blow-up}, we can show  
\begin{equation} \label{eq;act-u-r}
\A^{\lmd_n}(\bfu^{\lmd_n}; 0, \lmd_n\dl) < \A^{\lmd_n}(\bfr^{\lmd_n}; 0, \lmd_n\dl)\;\; \text{ for } n \text{ large enough}.
\end{equation}
Again for each $n$, we can define a $\bfv_n \in H^1_{0, \dl}(\rr^d)$ as below, 
$$ \bfv_n(t) = \lmd_n^{-\se} \bfu^{\lmd_n}(\lmd_n t). $$
Since the $\lmd_n$-blow-up of $\bfv_n$ is nothing but $\bfu^{\lmd_n}$, by \eqref{eq;XiLmdn-2}, 
$$ \frac{\bfv_n(0)}{|\bfv_n(0)|} = \frac{\tilde{\bfr}(0)}{|\tilde{\bfr}(0)|} =\bs;
$$
$$ \bfv_n(\dl) + q_N( \dl)= \lmd_n^{-\se} \bfr^{\lmd_n}( \lmd_n \dl) + q_N( \dl)=\bfr(  \dl) + q_N( \dl) = x( \dl) . $$
Meanwhile by \eqref{eq;act-u-r} and Lemma \ref{lem:ScaleofAction-2}  
$$ \begin{aligned}
\A(x; 0, \dl) & - \A(\bfv_n + q_N; 0, \dl) = \A(\bfr; 0, \dl)  - \A(\bfv_n; 0, \dl) \\
& = \lmd_n^{-\sy} \left( \A^{\lmd_n}(\bfr^{\lmd_n}; 0, \lmd_n\dl) - \A^{\lmd_n}(\bfu^{\lmd_n}; 0, \lmd_n\dl) \right)>0.
\end{aligned} $$
This is a contradiction to the fact that $x$ is a local minimizer in $B_{t_1, t_2}(X_1, X_2)$. 

The perpendicular property follows from the natural boundary condition obtained from the first variation $\delta_{\xi} \action$ of the action functional at $x$. More precisely, let $\xi \in B_{t_1, t_2}(X_1, X_2)$ be any admissible variation of $x$ which is supported in a small neighborhood of $t_1$ in $[t_1, t_2]$. The fact that $\dl_{\xi}\A(x)$ vanishes and $x$ satisfies \eqref{eq;RNBP1} implies 
$$ \lim_{t \to t_1^+} \dot{x}(t) \cdot \xi(t) = \dl_{\xi} \A(x) =0.
$$
As a result, $\dot{x}(t_1^+) \perp X_1$. Then a similar argument can show $\dot{x}(t^-_2) \perp X_2$.

(b). When $x(t_1)$ is collision-free, the perpendicular property follows the same argument as in the proof of property (a). Now let's assume $x(t_1) = q_{i_k}(t_1)$ is a collision. In this case as long as $\bs \ne -\bs^+$, where 
$$ \bs^+ = \lim_{t \to t_1^+} \frac{x(t)-q_{i_k}(t)}{|x(t) - q_{i_k}(t)|}, \;\; \bs = \frac{q_{i_{3-k}}(t_1)-q_{i_k}(t_1) }{|q_{i_{3-k}}(t_1)-q_{i_k}(t_1) |} $$
by the same argument as in the proof of property (a), we can always find a small enough variation of $x$ inside $B_{t_1, t_2}(X_1, X_2)$ whose action value is strictly smaller. 

(c). The proof is similar as property (b).

(d). When $x(t_1)$ is collision-free, the perpendicular property follows the same argument as in the proof of property (a). Now let's assume $x(t_1) = q_{i}(t_1)$ is a collision. Then as long as $\bs \ne -\bs^+$, where 
$$ \bs^+ = \lim_{t \to t_1^+} \frac{x(t)-q_{i}(t)}{|x(t) - q_{i}(t)|}, \;\; \bs = \frac{p -q_{i}(t_1) }{|p - q_{i}(t_1)|}, \;\; \forall p \in X_1 \setminus \{q_i(t_1)\} $$
by the same argument as in the proof of property (a), we can always find a small enough variation of $x$ inside $B_{t_1, t_2}(X_1, X_2)$ whose action value is strictly smaller. 

(e). The proof is similar as property (d).

\end{proof}

\section{Applications: unbounded solutions}  \label{sec:unbounded}

In this section we apply the Marchal-type lemmas to study unbounded solutions in the restricted problem, more precisely give proofs of Theorem \ref{thm;hyperbolic} and \ref{thm;parabolic}. We follow the idea given by Polimeni and Terracini \cite{PT24}, where they proved the existence of hyperbolic, parabolic and hyperbolic-parabolic solutions of the $N$-body problem in a uniform approach by introducing suitable renormalized action functionals. Throughout this section we always assume the motion of the primaries are bounded and collision-free.


\subsection{Hyperbolic solutions} \label{sub;hyp-sol} We will look for solutions $x: [t_0, \infty) \to \rr^d$ of \eqref{eq;RNBP1} satisfying 
\begin{equation}
    \label{eq;xt-Hy} x(t_0) =x_0 \text{ and }  x(t) = \xb(t) +\vph(t) + \xs, \;\;  \forall t \ge t_0,
\end{equation}
where $r_0$ is as defined in \eqref{eq;r0} and
 \begin{equation}
   \label{eq;xb-Hy} \xs=  x_0 - (\sqrt{2h} + r_0)\bs \; \text{ and } \;  \xb(t)  =  \sqrt{2h}(t -t_0 + 1)\bs + r_0 \bs.
 \end{equation}
 
\begin{prop}
    \label{prop;phi-sol-hp} Such an $x$ is a hyperbolic solution of \eqref{eq;RNBP1}, if and only if $\vph$ is a solution of  
    \begin{equation}
        \label{eq;RNBP-2} \ddot{\vph} = \nabla U(\xb + \vph + \xs, t) \text{ and } \vph(t_0)=0.
    \end{equation}
\end{prop}
This follows from a direct computation. Moreover the above equation is the Euler-Lagrange equation of the following renormalized action functional
\begin{equation}
    \Ah(\vph) = \int_{t_0}^{\infty} \ey |\dot{\vph}|^2 + U(\xb + \vph + \xs, t) - U(\xb, t) \,dt.
\end{equation} 
We subtract $U(\xb, t)$ to make the value of improper integral finite.

The proper function space to study $\Ah$ is $\D_0$ defined as below 
$$ \D_0=\D_{0}^{1,2}:=\{ \vph \in H^1_{loc}([t_0, \infty), \rr^d): \vph(t_0)=0 \text{ and  } \int_{t_0}^{\infty} |\dot{\vph}|^2 \,dt < \infty \}
$$
with the following norm and inner product 
$$ \|\vph\|_{\D_0} = \left( \int_{t_0}^{\infty} |\dot{\vph} |^2 \,dt \right)^{\ey}, \;\; \la \vph, \psi \ra_{\D_0}: = \int_{t_0}^{\infty} \la \dot{\vph}, \dot{\psi} \ra \,dt. $$

\begin{prop}
    \label{prop;Hardy-ineq} $\D_0$ is a Hilbert space, such that given any $\vph \in \D_0$ 
    $$ |\vph(t)|  \le \|\vph\|_{\D_0} \sqrt{t-t_0 +1}, \;\; \forall t \ge t_0; $$
$$     \int_{t_0}^{\infty} \frac{|\vph|^2}{|t -t_0 +1|^2} \,dt \le 4 \int_{t_0}^{\infty} |\dot{\vph}|^2 \,dt. $$
\end{prop}
A detailed proof of this proposition can be found in \cite{BDFT21}. 
\begin{prop}
    \label{prop;MiniHyper:label}
    $\Ah$ has a global minimizer $\vph(t)$ in $\D_0$. 
\end{prop}

\begin{proof}
    By a standard argument, it is enough to show $\Ah$ is both coercive and weakly lower semi-continuous in $\D_0$. For coercivity, it is enough to show that $\Ah(\vph) \to \infty$, if $\|\vph\|_{\D_0} \to \infty$. By \eqref{eq;r0}, \eqref{eq;xt-Hy}, \eqref{eq;xb-Hy} and Proposition \ref{prop;Hardy-ineq}, for any $t \ge t_0$, 
$$ |\xb(t) -q_i(t)| \ge |\xb(t)| - |q_i(t)| \ge \sqrt{2h}(t-t_0 +1);$$
$$ 
 |\xb(t) -q_i(t) + \vph(t) + \xs| \le \sqrt{2h}(t-t_0 +1) + \|\vph\|_{\D_0}\sqrt{t- t_0+1} + C_1,
$$
throughout this section $C_i$, $i =1, 2, \dots$, will be positive constants that vary from line to line, but always independent of $\vph$ and $\vph_n$.  

As a result, for any $t >t_0$, 
$$ 
  U(\xb, t) - U(x, t) \le  \frac{m}{\sqrt{2h}(t -t_0 +1)} - \frac{m}{\sqrt{2h}(t-t_0 +1) + \|\vph\|_{\D_0}\sqrt{t-t_0+1} + C_1}. 
$$ 
We first make a change of the time variable: $t -t_0 +1= s^2$ with $s > 0$,  
$$ \int_{t_0}^{\infty} U(\xb, t) - U(x, t) \,dt \le \frac{2m}{\sqrt{2h}} \int_1^{\infty} \left( \frac{1}{s^2  } - \frac{1}{s^2 + \frac{\|\vph\|_{\D_0}}{\sqrt{2h}}s + \frac{C_1}{\sqrt{2h}}} \right) s\,ds. $$
Then make a further change by $\tau = \frac{2\sqrt{2h}}{\|\vph\|_{\D_0}}s$. As a result, when $\|\vph\|_{\D_0}$ is large enough 
$$
\begin{aligned}
    \int_{t_0}^{\infty} U(\xb, t) - U(x, t) \,dt & \le \frac{2m}{\sqrt{2h}} \int_{\frac{2\sqrt{2h}}{\|\vph\|_{\D_0}}}^{\infty}  \frac{1}{\tau} - \frac{\tau}{\left(\tau + 1 \right)^2} \,d\tau.
\end{aligned}
$$

Meanwhile notice that for any $0 < \ep <1$, 
$$ \int_{\ep}^1  \frac{1}{\tau} - \frac{\tau}{(\tau+1)^2} \,d\tau = \int_{\ep}^{1} \frac{2 + \tau^{-1}}{(\tau + 1)^2} \,d\tau \le \int_{\ep}^1 2 +\tau^{-1} \,d\tau \le 2 -\ln \ep;
$$
$$ \int_1^{\infty}  \frac{1}{\tau} - \frac{\tau}{(\tau+1)^2} \,d\tau =  \int_{1}^{\infty} \frac{2 +\tau^{-1}}{(\tau+1)^2} \,d\tau \le \int_1^{\infty} \frac{3}{\tau^2} \,d\tau = 3.
$$
Hence for $ \|\vph\|^2_{\D_0} $ large enough,   
$$ 
 \frac{\sqrt{2h}}{2m} \int_{t_0}^{\infty} U(\xb, t) - U(x, t) \,dt \le \int_{\frac{2\sqrt{2h}}{\|\vph\|_{\D_0}}}^{1} + \int_1^{\infty}  \frac{1}{\tau} - \frac{\tau}{(\tau+1)^2} \,d\tau \le \ln \|\vph\|_{\D_0} + 5 - \ln 2\sqrt{2h}.
$$
This implies the coercivity of $\Ah$, as when $\|\vph\|_{\D_0} \to \infty$, 
$$ \Ah(\vph) \ge \ey \|\vph\|^2_{\D_0} - \ln \|\vph\|_{\D_0} -5 + \ln 2 \sqrt{2h} \to \infty. $$

To prove the weakly lower semi-continuity of $\Ah$. Let's choose a bounded sequence $\{\vph_n \}$ in $\D_0$, i.e., $\|\vph_n\|_{\D_0} \le C_2$ for all $n$, which converges weakly to a $\vph \in \D_0$. By passing to a proper subsequence, $\vph_n(t)$ will converge to $\vph(t)$ pointwisely for almost every $t \in (t_0, \infty)$ and uniformly on $[t_0, T]$ for any fixed $T>t_0$. We may further assume $\{\Ah(\vph_n) \}$ is bounded as well.  Let 
$$ x_n= \xb + \vn + \xs, \; x = \xb + \vph + \xs. $$
For any $s \in [0, 1]$, by Proposition \ref{prop;Hardy-ineq},  
\begin{equation}
    \begin{aligned}
        |\xb - q_i + s(\vph_n + \xs)| & \ge |\xb - q_i| - s|\vph_n| - s |\xs| \\
        & \ge \sqrt{2h}(t-t_0 + 1) -s C_2\sqrt{t-t_0+1} - s |\xs|>0 
    \end{aligned}
\end{equation}
for any $t \ge T$, when $T$ is large enough. This allows us to write
\begin{equation*}
\label{eq;Ui-dff-ubd} 
    \frac{1}{|x_n-q_i|} - \frac{1}{|\xb-q_i|} = \int_0^1 \frac{d}{ds} \frac{1}{|\xb-q_i +s(\vph_n +\xs)|} \,ds  = - \int_0^1 \frac{\la \xb -q_i +s(\vph_n + \xs), \vph_n + \xs \ra}{|\xb-q_i +s(\vph_n +\xs)|^{3}} \,ds.
\end{equation*}
As a result, 
\begin{equation}
\label{eq;U-diff-ubd} \begin{aligned}
  |U(x_n, t)- U(\xb, t)| & \le \sum_i m_i \int_0^1 \frac{|\vph_n + \xs|}{|\xb-q_i +s(\vph_n +\xs)|^{2}} \,ds \\
  & \le m \int_0^1 \frac{C_2 \sqrt{t-t_0 +1} + |\xs|}{(\sqrt{2h}(t -t_0 +1) -sC_2\sqrt{t-t_0 +1} - s|\xs|)^2} \,ds\\
  & \le \frac{m}{|t-t_0+1|^{\frac{3}{2}}} \int_0^1 \frac{C_2 + \frac{|\xs|}{\sqrt{t-t_0 +1}}}{\left[ \sqrt{2h} - s\left( \frac{C_2}{\sqrt{t-t_0 +1}} + \frac{|\xs|}{t-t_0 +1} \right) \right]^2} \,ds \\
  & \le C_3(t-t_0+1)^{-\frac{3}{2}}.
\end{aligned}
\end{equation}
The last inequality holds, as the last integral is finite and has a uniform upper bound, for all $t \ge T$. Since $ \int_{T}^{\infty} (t-t_0+1)^{-\frac{3}{2}} \,dt$ is finite, by the dominated convergence theorem 
$$ \lim_{n \to \infty} \int_{T}^{\infty} U(x_n, t)- U(\xb, t) \,dt = \int_T^{\infty} U(x, t) - U(\xb, t) \,dt < \infty. $$

Again by the above estimates,   
$$ \int_{t_0}^T U(x_n, t) - U(\xb, t) \,dt \le \Ah(\vn) - \int_{T}^{\infty} U(x_n, t) -U(\xb, t)\,dt \le C_4. $$
Then
$$ \int_{t_0}^T U(x_n, t)\,dt  \le C_4 + \int_{t_0}^{T}U(\xb, t) \,dt  \le C_4 + \int_{t_0}^T \frac{m}{\sqrt{2h}(t-t_0+1)} \,dt  \le C_5. 
 $$ 
Since $U(x_n, t)$ is always positive, by Fatou's lemma, 
$$ \int_{t_0}^{T} U(x, t) \,dt \le \liminf_{n \to \infty} \int_{t_0}^T U(x_n, t)\,dt. $$
The weakly lower semi-continuity of $\Ah$ now follows from the above estimates and the convexity of $|\dot{\vph}_n|^2$. 

\end{proof}

\begin{prop}
\label{prop;NormalMinCollFree-hyper} If $\vph \in \D_0$ is a minimizer of $\Ah$, then $x =\xb +\vph +\xs$ is a fixed-end minimizer of $\A$ for any $t_1 < t_2$, and a collision-free solution of \eqref{eq;RNBP1}. Moreover  
$$ \vph(t) = \frac{m}{2h} \bs \ln t  + o(\ln t), \;\; t \to \infty. $$  

\end{prop}

\begin{proof}
    By a contradiction argument, let's assume there is a $y \in H^1([t_1, t_2], \rr^d)$ satisfying $y(t_i)=x(t_i)$, $i =1, 2$, and $\A(y; t_1, t_2) < \A(x; t_1, t_2)$. Then the following path $\psi$ belongs to $\D_0$
    $$ \psi(t) = \begin{cases} \vph(t), &  \text{ if } t \notin [t_1, t_2]; \\
    y(t) - \xb(t) -\xs, & \text{ if } t \in [t_1, t_2].  
    \end{cases}
    $$
    Then we get the following inequality, which is a contradiction to the fact that $\vph$ is a minimizer
    $$  \Ah(\vph) -\Ah(\psi) = \A(x; t_1, t_2) - \A(y; t_1, t_2) > 0.$$
    By Theorem \ref{thm;Marchal-restricted}, $x$ is a collision-free solution of \eqref{eq;RNBP1} and $\vph(t)$ is a solution of \eqref{eq;RNBP-2}. Hence
     $$ \begin{aligned}
         \lim_{t \to \infty} t^2 \ddot{\vph}(t) & = -\lim_{t \to \infty} t^2 \sum_i \frac{m_i(\xb(t) -q_i(t) + \vph(t) + \xs)}{|\xb(t) -q_i(t) + \vph(t) + \xs|^3}\\
         & = -\frac{1}{2h} \lim_{t \to \infty} \sum_i \frac{m_i \left(\bs + \frac{r_0 - \sqrt{2h}t_0 + \sqrt{2h}}{\sqrt{2h}t} + \frac{\vph(t)+ \xs-q_i(t)}{\sqrt{2h}t} \right)}{ \left|\bs + \frac{r_0 - \sqrt{2h}t_0 + \sqrt{2h}}{\sqrt{2h}t} + \frac{\vph(t)+ \xs-q_i(t)}{\sqrt{2h}t}  \right|^3} \\
         & = -\frac{m}{2h}\bs.
     \end{aligned}
     $$
     This implies the following estimate and the desired result follows sequentially
     $$ \ddot{\vph}(t) = -\frac{m}{2h}t^{-2} \bs +o(t^{-2}), \; \text{ as } t \to \infty. $$ 
  \end{proof}
  Theorem \ref{thm;hyperbolic} now follows directly from the previous three propositions.

\subsection{Parabolic solutions} \label{sub;para-sol}
We will search for solutions $x: [t_0, \infty) \to \rr^d$ of \eqref{eq;RNBP1} satisfying 
\begin{equation}
    \label{eq;xt} x(t_0) =x_0 \text{ and } x(t) = \xb(t) +\vph(t) + \xs, \;\; \forall t \ge t_0, 
\end{equation}
where  
\begin{equation}
    \label{eq;xb-Pa} \xs=  x_0 - (\al +r_0)\bs \; \text{ and} \; \xb(t)  =  \al(t-t_0 +1)^{\se} \bs + r_0 \bs. 
\end{equation}

\begin{prop}
    \label{prop;phi-sol-pa} Such an $x$ is a parabolic solution of \eqref{eq;RNBP1}, if and only if $\vph$ is a solution of  
    \begin{equation}
        \label{eq;RNBP-3} \ddot{\vph} = \nabla U(\xb + \vph + \xs, t) -\ddot{\xb} \text{ and } \vph(t_0)=0.
    \end{equation}
\end{prop}
This is the Euler-Lagrange equation of the following renormalized action functional
$$ \Ap(\vph) = \int_{t_0}^{\infty} \ey |\dot{\vph}|^2 + U(x, t) - U(\xb, t) - \la \ddot{\xb}, \vph \ra \,dt. $$

The proof of coercivity is more involved here. To get needed estimates, we follow \cite{PT24} and set 
\begin{equation} \label{eq;ui-vi}
 u_i = |\xb-q_i|^2, \; v_i  = |\xb-q_i + \vph + \xs|^2 - u_i, \; \forall i=1, \dots, N.
\end{equation} 
By Taylor's expansion, 
\begin{equation}
\begin{aligned}
\frac{1}{|x -q_i|} &- \frac{1}{|\xb -q_i|}  = (u_i + v_i)^{-\ey}- u_i^{-\ey}= -\ey u_i^{-\frac{3}{2}}v_i + R(u_i, v_i)\\
& =-\ey \frac{|\vph|^2}{|\xb -q_i|^3} - \ey \frac{|\xs|^2}{|\xb -q_i|^3} - \frac{ \la \xb -q_i, \vph\ra }{|\xb -q_i|^3} - \frac{ \la \xb -q_i, \xs \ra }{|\xb -q_i|^3} - \frac{\la \vph, \xs \ra}{|\xb -q_i|^3} + R(u_i, v_i)
\end{aligned}
\end{equation}
where 
\begin{equation}
\label{eq;R-uivi} R(u_i, v_i) = \int_0^1 \int_0^1 \frac{3}{4} t_2(u_i + t_1 t_2 v_i)^{-\frac{5}{2}}v_i^2 \,dt_1  \,dt_2 \ge 0. 
\end{equation}

As a result, 
\begin{equation}
    \label{eq;Ux-Uxb} \begin{aligned}
         U(x, t) &- U(\xb, t)  - \la \ddot{\xb}, \vph \ra = \sum_i m_i R(u_i, v_i) -\ey \sum_{i} \frac{m_i |\vph|^2}{|\xb -q_i|^3} - \sum_{i} \frac{m_i\la \vph, \xs \ra}{|\xb -q_i|^3}  \\
         & - \left( \sum_{i} \frac{m_i \la \xb -q_i, \vph \ra}{|\xb -q_i|^3} + \la \ddot{\xb}, \vph \ra \right)  - \left( \ey \sum_{i} \frac{m_i |\xs|^2}{|\xb -q_i|^3} + \sum_{i} \frac{m_i \la \xb -q_i, \xs \ra}{|\xb -q_i|^3} \right).
    \end{aligned}
\end{equation}
Notice that by \eqref{eq;xb-Pa}, for any $i$ and $t \ge t_0$, 
$$ |\xb(t) - q_i(t)| \ge \al (t - t_0 +1)^{\se}. $$
Then by Proposition \ref{prop;Hardy-ineq}, 
\begin{equation}
    \label{eq;U-upb-1} \ey \int_{t_0}^{\infty} \sum_{i} \frac{m_i |\vph|^2}{|\xb -q_i|^3} \,dt \le \frac{m}{2 \al^3} \int_{t_0}^{\infty} \frac{|\vph|^2}{(t-t_0 + 1)^2} \,dt \le \frac{4}{9} \int_{t_0}^{\infty} |\dot{\vph}|^2 \,dt;
\end{equation} 
\begin{equation}
    \label{eq;U-upb-2} \int_{t_0}^{\infty}  \sum_{i} \frac{m_i |\la \vph, \xs \ra|}{|\xb -q_i|^3}  \,dt \le \frac{m |\xs|}{\al^3} \int_{t_0}^{\infty} \frac{\|\vph\|_{\D_0}}{(t- t_0 +1)^{\frac{3}{2}}} \,dt = \frac{4 |\xs| }{9}\|\vph\|_{\D_0};
\end{equation}
\begin{equation}
    \label{eq;U-upb-3} \begin{aligned}
    \int_{t_0}^{\infty} \ey \sum_{i} \frac{m_i |\xs|^2}{|\xb -q_i|^3} & + \sum_{i} \frac{m_i \la \xb -q_i, \xs \ra}{|\xb -q_i|^3} \,dt \\
    & \le \int_{t_0}^{\infty} \frac{m |\xs|^2}{2 \al^3 (t-t_0 +1)^2} + \frac{m |\xs|}{\al^2 (t-t_0 +1)^{\frac{4}{3}}} \,dt \le C_1.
    \end{aligned}
\end{equation}
For the remaining terms in \eqref{eq;Ux-Uxb}, notice that for each $i$, 
\begin{equation}
    \label{eq;qt-upb} |\qt_i(t)|\le \frac{2r_0}{\al}(t-t_0 + 1)^{-\frac{2}{3}}, \; \text{ where } \tilde{q}_i(t) = \frac{q_i(t)-r_0 \bs}{\al(t-t_0 +1)^{\se}}.
\end{equation}
Again by Taylor's expansion,  
$$ |\bs-\qt_i|^{-3} = 1 - \frac{3}{2} ( |\qt_i|^2 - 2 \la \qt_i, \bs \ra ) +  R(\qt_i, \bs). $$
where 
\begin{equation}
\label{eq;R-qt} R(\qt_i, \bs) = \int_0^1 \int_0^1   \frac{15 ( |\qt_i|^2 - 2 \la \qt_i, \bs \ra)^2}{ 4 [1 + t_1 t_2 ( |\qt_i|^2 - 2 \la \qt_i, \bs \ra)]^{\frac{7}{2}}} t_2 \,dt_1 dt_2.
\end{equation}
\begin{lem} \label{lem;R}
    There is a $T>t_0$ and a constant $C_2>0$, such that  
    $$R(\qt_i(t), \bs) \le C_2|\qt_i(t)|^2, \; \forall t \ge T.$$
\end{lem}
\begin{proof}
By \eqref{eq;qt-upb}, there is a $T>t_0$, such that for each $i$
$$ |\qt_i(t)| \le 1/4, \; \forall t \ge T. $$
As a result, for any $0 \le t_1, t_2 \le 1$, 
    $$ \begin{aligned}
        1 + t_1t_2(|\qt_i|^2 - 2 \la \qt_i, \bs \ra) & \ge 1 -t_1 t_2 + t_1t_2|\bs -\qt_i|^2 \ge 1 -t_1 t_2 +t_1 t_2 (1 - |\qt_i|)^2 \\
        & \ge 1 - 2t_1 t_2 |\qt_i| \ge 1 - \ey t_1t_2 \ge \ey.
    \end{aligned}
    $$
    Then
    $$   R(\qt_i(t), \bs)  \le \frac{15}{4}|\qt_i(t)|^2 \int_0^1 \int_0^1  \frac{2 +|\qt_i(t)|}{(1 -t_1 t_2 + t_1t_2|\bs -\qt_i(t)|^2)^{\frac{7}{2}}} \,dt_1 dt_2  \le C_2 |\qt_i(t)|^2. 
    $$
\end{proof}
Meanwhile by a direct computation
$$ \sum_i m_i\frac{ \xb(t)- q_i(t)}{|\xb(t)-q_i(t)|^3} = \frac{1}{\al^2 (t-t_0 +1)^{\frac{4}{3}}} \sum_i m_i \frac{\bs -\qt_i(t)}{|\bs-\qt_i(t)|^3}$$
where
\begin{equation} \label{eq;R-exp}
    \begin{aligned}
        \sum_{i } & m_i \frac{\bs -\qt_i}{|\bs -\qt_i|^3}  = m\bs + 3 \sum_{i } m_i \la \qt_i, \bs \ra \bs - \sum_{i} m_i\qt_i \\
        & \quad - \frac{3}{2} \sum_i m_i |\qt_i|^2 \bs + \sum_i m_i \left( \frac{3}{2}|\qt_i|^2 - 3 \la \qt_i, \bs \ra \right) \qt_i + \sum_i m_i R(\qt_i, \bs)(\bs - \qt_i).
    \end{aligned}
\end{equation}
By Lemma \ref{lem;R}, for any $t \ge T$, 
\begin{equation}
    \label{eq;ddot-xb-upb}  \left|\sum_i \frac{m_i(\xb -q_i)}{|\xb-q_i|^3} + \ddot{\xb} \right|   \le C_3(t -t_0 +1)^{-2}.
\end{equation}
Combining this with the Cauchy-Schwarz inequality and Proposition \ref{prop;Hardy-ineq}, 
\begin{equation} \label{eq;U-upb-4}
    \int_{T}^{\infty}   \left|  \la\sum_i \frac{m_i(\xb -q_i)}{|\xb-q_i|^3} + \ddot{\xb}, \vph \ra  \right|  \,dt  \le C_3\int_{t_0}^{\infty} \frac{|\vph|}{|t-t_0 +1|^2} \,dt  \le C_4\|\vph\|_{\D_0}.
\end{equation}
Meanwhile for $t \in [t_0, T]$, 
\begin{equation} \label{eq;U-upb-5}
    \int_{t_0}^{T} \left|  \la\sum_i \frac{m_i(\xb -q_i)}{|\xb-q_i|^3} + \ddot{\xb}, \vph \ra  \right| \,dt  \le C_5 \int_{t_0}^{\infty} \frac{|\vph|}{|t-t_0+1|^{\frac{4}{3}}} \,dt \le C_6 \|\vph\|_{\D_0}.
\end{equation}
\begin{prop}
    \label{prop;MiniPara:label}
    $\Ap$ has a global minimizer $\vph(t)$ in $\D_0$. 
\end{prop}

\begin{proof}
The coercivity of $\Ap$ in $\D_0$ follows from  \eqref{eq;U-upb-1}, \eqref{eq;U-upb-2}, \eqref{eq;U-upb-3}, \eqref{eq;U-upb-4} and \eqref{eq;U-upb-5}, as they imply
$$ \Ap(\vph) \ge \frac{1}{18} \|\vph\|^2_{\D_0}- C_1\|\vph\|_{\D_0}-C_2 \to \infty, \text{ if  } \|\vph\|_{\D_0} \to \infty. $$

To prove the weakly lower semi-continuity. Let's choose a bounded sequence $\{\vph_n \}$ in $\D_0$ with the same properties as in the proof of Proposition \ref{prop;MiniHyper:label}.  

First we will show the following integral is weakly lower semi-continuous for $T$ large enough
\begin{equation} \label{eq;int-T}
\int_T^{\infty} \ey |\dot{\vph}_n|^2 + U(x_n, t)- U(\xb, t)- \la \ddot{\xb}, \vph_n \ra \,dt.
\end{equation}
By Taylor expansion, for any $t \ge T$,
$$ U(x_n, t) = U(\xb, t) + \la \nabla U(\xb,t), \vph_n + \xs \ra + \ey \la \nabla^2 U(\xb, t)(\vph_n + \xs), \vph_n + \xs \ra + R(\xb, \vph_n+\xs, t)$$
where $R(\xb, \vph_n +\xs, t)$ is the remainder given by the following triple integral
\begin{equation*}
    \label{eq;R}  \int_0^1 \int_0^1 \int_0^1 \la \nabla^3 U(\xb(t) + t_1 t_2 t_3 (\vph_n(t)+ \xs), t)(\vph_n(t)+\xs), \vph_n(t)+\xs, \vph_n(t)+\xs \ra t_1 t_2^2 \,dt_1 dt_2 dt_3.
\end{equation*}
Then we can rewrite \eqref{eq;int-T} as
$$ \int_T^{\infty}  \ey |\dot{\vph}_n|^2 + U(x_n, t)- U(\xb, t)- \la \ddot{\xb}, \vph_n \ra \,dt =  B_1(\vph_n) + B_2(\vph_n) +B_3(\vph_n)$$
where 
$$ B_1(\vph_n)= \int_{T}^{\infty} \la \nabla U(\xb, t), \vph_n + \xs \ra - \la \ddot{\xb}, \vph_n \ra + \la \nabla^2 U(\xb, t)\xs, \vph_n \ra + \ey \la \nabla^2 U(\xb, t) \xs, \xs \ra \,dt; $$
$$ B_2(\vph_n) =  \int_{T}^{\infty} R(\xb, \vph_n + \xs, t) \,dt;$$
$$ B_3(\vph_n) =  \int_{T}^{\infty} \ey |\dot{\vph}_n|^2 + \ey \la \nabla^2 U(\xb, t)\vph_n, \vph_n\ra \,dt. $$
By a direct computation
\begin{equation}
    \label{eq;nabla2-U} \nabla^2 U(\xb, t) = 3 \sum_{i} m_i \frac{(\xb -q_i)(\xb -q_i)^T}{|\xb -q_i|^5}- \sum_i \frac{m_i I}{|\xb -q_i|^3}
\end{equation}
where $I$ is the identity matrix and $(\xb-q_i)^T$ is the transpose of $(\xb-q_i)$. 

By \eqref{eq;ddot-xb-upb} and \eqref{eq;nabla2-U}, for any $t \ge T$
$$ \left| \la \nabla^2 U(\xb, t) \xs, \xs \ra  \right| \le \frac{C_1}{|t -t_0 +1|^2},
$$  
$$ |\la \nabla U(\xb, t) - \ddot{\xb} + \nabla^2 U(\xb, t)\xs, \vph_n \ra| \le \frac{C_1|\vph_n(t)|}{|t-t_0 +1|^{2}} \le \frac{C_2}{|t-t_0 +1|^{\frac{3}{2}}}. 
$$
Hence both are integrable from $T$ to $\infty$ and by the dominated convergence theorem 
$$  \lim_{n \to \infty} B_1(\vph_n) = B_1(\vph). $$

For $B_2(\vn)$, notice that for any $t \ge T$,
$$ \begin{aligned}
    |\la \nabla^3 U(\xb(t) & + t_1 t_2 t_3 (\vph_n(t)+ \xs), t)(\vph_n(t)+\xs), \vph_n(t)+\xs, \vph_n(t)+\xs \ra | \\
    & \le \frac{C_1|\vph_n(t)+ \xs|^3}{|t-t_0 +1|^{\frac{8}{3}}} \le \frac{C_2|t-t_0 +1|^{\frac{3}{2}} + C_3}{|t-t_0+1|^\frac{8}{3}},
\end{aligned}
$$ 
which is also integrable from $T$ to $\infty$ and for the same reason as above 
$$ \lim_{n \to \infty} B_2(\vph_n) = B_2(\vph). $$

Meanwhile by \eqref{eq;nabla2-U}, for any $\psi \in \D_0$ and $t \ge T$, 
$$  -\sum_{i} \frac{m_i|\psi|^2}{|\xb -q_i|^3} \le    \la \nabla^2 U(\xb, t) \psi, \psi \ra  \le 4\sum_i \frac{m_i|\psi|^2}{|\xb -q_i|^3}. $$
As a result, 
$$ B_3(\psi) \le  \ey \|\psi\|_{\D_0}^2 + \frac{4}{9} \int_{t_0}^{\infty} \frac{|\psi(t)|^2}{|t-t_0 +1|^2} \,dt  \le C_1 \|\psi\|_{\D_0}^2; 
$$
$$ B_3(\psi) \ge \ey \|\psi\|_{\D_0}^2 -  \frac{1}{9} \int_{t_0}^{\infty} \frac{|\psi(t)|^2}{|t-t_0 +1|^2} \,dt \ge \frac{1}{18}\|\psi\|_{\D_0}^2.
$$
Notice that $\sqrt{B_3(\psi)}$ defines a norm on $\D_0$ and the above estimates show it is equivalent to $\| \cdot\|_{\D_0}$. Hence $B_3$ must be weakly lower semi-continuous with respect to $\|\cdot \|_{\D_0}$. 

Finally by a similar argument as in the proof for the hyperbolic case, 
$$ \begin{aligned}
    \int_{t_0}^T U(x_n, t)\,dt & \le \Ap(\vph_n) + \int_{t_0}^T U(\xb, t) \,dt + \int_{t_0}^T |\la \ddot{\xb}, \vph_n\ra| \,dt \\
    & \le C_1 +  \frac{m}{\al}  \int_{t_0}^T|t-t_0+1|^{-\se} \,dt + \frac{2\al}{9} \|\vph_n\|_{\D_0}\int_{t_0}^{T} |t-t_0 +1|^{-\frac{5}{6}} \,dt \le C_2. 
\end{aligned}
$$
Then by Fatou's lemma 
\begin{equation}
    \label{eq;U-upb-para-1} \int_{t_0}^T U(x, t) \,dt \le \liminf_{n \to \infty} \int_{t_0}^T U(x_n, t) \,dt. 
\end{equation}
As a result,
\begin{equation*}
    \int_{t_0}^T \ey |\dot{\vph}|^2 + U(x, t)- U(\xb, t)- \la \ddot{\xb}, \vph\ra \,dt \le \liminf_{n \to \infty} \int_{t_0}^T \ey |\dot{\vph}_n|^2 + U(x_n, t)- U(\xb, t)- \la \ddot{\xb}, \vph_n \ra \,dt
\end{equation*}

This finishes our proof of the weakly lower semi-continuity of $\Ap$. 
\end{proof}

By Proposition \ref{prop;Hardy-ineq} and the same argument as in the proof of Proposition \ref{prop;NormalMinCollFree-hyper}
\begin{prop}
\label{prop;NormalMinCollFree-para} If $\vph \in \D_0$ is a minimizer of $\Ap$, then $x =\xb +\vph +\xs$ is a fixed-end minimizer of $\A$ for any $t_1 < t_2$, and a collision-free solution of \eqref{eq;RNBP1}. Moreover $|\vph(t)| = O(\sqrt{t})$, when $t \to \infty$. 
\end{prop}
Theorem \ref{thm;parabolic} then follows directly from Propositions \ref{prop;MiniPara:label} and \ref{prop;NormalMinCollFree-para}.

\subsection{Bi-hyperbolic and bi-parabolic solutions}
To show the solutions are regularizable when there are collisions, we recall the following well-known result, for a proof see \cite{BOZ19} and \cite{Zhao23}.
\begin{prop}
    \label{prop;regularizble} Let $x|_{[t_0 -\dl, t_0+\dl]}$ be a collision-ejection solution of \eqref{eq;RNBP1} with an isolated binary collision at time $t_0$ with $x(t_0) = q_{i_0}(t_0)$. It is a regularizable solution when the following one-sided limits exist and satisfy
    $$ \lim_{t \to t_0^-} \frac{x(t)-q_{i_0}(t)}{|x(t)-q_{i_0}(t)|} =  \lim_{t \to t_0^+} \frac{x(t)-q_{i_0}(t)}{|x(t)-q_{i_0}(t)|};
    $$
    $$ \lim_{t \to t_0^-} \ey |\dot{x}(t) - \dot{q}_{i_0}(t)|^2 - U(x(t), t) = \lim_{t \to t_0^+} \ey |\dot{x}(t) - \dot{q}_{i_0}(t)|^2 - U(x(t), t). $$
\end{prop}

The idea of proving Theorem \ref{thm;bi-hyp-para} is again to look for minimizers of suitable renormalized action functionals, but with the initial configuration varying on some straight line segment. Then we can extend the minimizers to entire solutions through symmetry. Because of this, the admissible space of paths will be slightly different from $\D_0$. 

Given any subset $\Om \subset \rr^d$, we define
$$ \D(\Om): = \{ \vph \in H^1_{loc}([t_0, \infty), \rr^d): \; \vph(t_0) \in \Om \text{ and } \int_{t_0}^{\infty} |\dot{\vph}|^2 \,dt < \infty \}. 
$$
Introduce the following inner product in $\D(\Om)$

$$ \langle \vph, \psi \rangle_{\D} = \langle \vph(t_0), \psi(t_0) \rangle + \int_{t_0}^{\infty} \langle \vph(t), \psi(t) \rangle \,dt. 
$$
and let $\| \cdot \|_{\D}$ be the corresponding norm.

\begin{prop} \label{prop;D-Om}
When $\Om$ is bounded and closed in $\rr^d$, $\D(\Om)$ is a Hilbert space, and there is a finite constant $C>0$ only depending on $\Om$, such that for any $\vph \in \D(\Om)$,
    $$ |\vph(t)| \le C + \|\vph\|_{\D_0} \sqrt{t -t_0 +1}, \; \forall t \ge t_0;$$
    $$ \int_{t_0}^{\infty} \frac{|\vph(t) - \vph(t_0)|^2}{|t-t_0 +1|^2} \,dt \le 4 \int_{t_0}^{\infty} |\dot{\vph}|^2 \,dt;$$
\end{prop}
We omit the proof, as it is similarly to Proposition \ref{prop;Hardy-ineq} and the details can be found in \cite{PT22}. 

\begin{proof}[Proof of Theorem \ref{thm;bi-hyp-para}]
    
First let's consider bi-hyperbolic solutions. Fix an arbitrary $1 \le i \le N-1$, let's assume the end points of $\Om_i(V)$ are $q_{i_1}(t_0)$ and $q_{i_2}(t_0)$. Then set 
$$ \xb(t) = \sqrt{2h}(t -t_0 +1) \bs + r_0 \bs, \; \forall t \ge t_0. $$
and 
$$ \bar{\Om}_i := \{ y \in \rr^d: \; y + \xb(t_0) \in \Om_i(V) \}. $$

Define the admissible space of paths as 
$$ \D_i := \D(\bar{\Om}_i(V)) = \{ \vph \in H^1_{loc}([t_0, \infty), \rr^d): \;  \vph(t_0) \in \bar{\Om}_i, \int_{t_0}^{\infty} |\dot{\vph}|^2 \,dt < \infty \}. 
$$
With Proposition \ref{prop;D-Om}, by a similar argument as in the proof of Proposition \ref{prop;MiniHyper:label} we can show the following action functional $\Ah$ has a global minimizer in $\D_i$
    $$ \Ah(\vph) = \int_{t_0}^{\infty} \ey |\dot{\vph}|^2 + U(\xb + \vph, t) - U(\xb, t) \,dt. $$

Let $\vph_i \in \D_i$ be such a global minimizer of $\Ah$ in $\D_i$ and $x_i = \xb + \vph_i$, then by the exact same argument as in the proof of Proposition \ref{prop;NormalMinCollFree-hyper}, we can show that $x_i|_{[t_1, t_2]}$ is a fixed-end minimizer of $\A$, for any $t_2 >t_1 >t_0$. Meanwhile by the definition of $\D_i$, for any $t_1>t_0$, $x_i|_{[t_0, t_1]}$ must be a minimizer of $\A$ in $B_{t_0, t_1}(\Om_i(V), x_i(t_1))$. Then by Theorem \ref{thm;Marchal-rest-2}, $x_i(t_0)$ is either collision-free with $\dot{x}_i(t_0^+)$ perpendicular to $\Om_i(V)$, or $x_i(t_0) = q_{i_j}(t_0)$, for $j =1, 2$ and  
\begin{equation}
    \label{eq;lim-xi}  \lim_{t \to t_0^+} \frac{ x_i(t)- q_{i_j}(t) }{|x_i(t)- q_{i_j}(t)|} = - \frac{q_{i_{3-j}}(t_0)- q_{i_j}(t_0)}{|q_{i_{3-j}}(t_0) -q_{i_j}(t_0)|}.
\end{equation}

Now extend the domain of $x_i$ to $(-\infty, t_0]$ by setting
$$ x_i(t) = R_V x_i(2t_0 -t), \;\; \forall t \le t_0. $$
By Proposition \ref{prop;regularizble}, $x_i|_{\rr}$ is either a collision-free or a regularizable solution of \eqref{eq;RNBP1}.

Let $x_j$ be a bi-hyperbolic solution obtained by a similar argument as above but with $x_j(t_0) \in \Om_j(V)$, for any $1 \le j \ne i \le N-1$. When $\Om_i(V) \cap \Om_j(V) = \emptyset$, it is obvious $x_i$ and $x_j$ are different. When the intersection is non-empty, we must have $\Om_i(V) \cap \Om_j(V) = q_{i_1}(t_0)$ or $q_{i_2}(t_0)$. In this case the only way that $x_i$ and $x_j$ are the same is that they both collide with the same primary at $t_0$, but then \eqref{eq;lim-xi} implies they must approach the primary from the opposite directions, so can not be the same solution.

The proof for bi-parabolic solutions is almost the same as above after replacing $\Ah$ by 
$$ \Ap(\vph) = \int_{t_0}^{\infty} \ey |\dot{\vph}|^2 + U(\xb + \vph, t) - U(\xb, t) - \langle \ddot{\xb}, \vph \rangle \,dt$$
with
$$ \xb(t) = \al (t - t_0 + 1)^{\se} \bs + r_0 \bs, \;\; \forall t \ge t_0.$$

\end{proof}
\section{Applications: bounded solutions } \label{sec;bounded}

In this section we will first prove the theorems about bounded solutions given in the introduction. Then we will give all kinds of examples that satisfy the conditions required by these theorems. 

\begin{proof}[Proof of Theorem \ref{thm;periodic-1}]

Since $\A$ is weakly lower-semi continuous, to show $\A$ has a minimizer in $B_{t_1, t_2}(V_1, V_2)$, by the standard argument in the direct methods of calculus of variations, it is enough to prove the coercivity of $\A$.   

Given any $x \in B_{t_1, t_2}(V_1, V_2)$, as $V_1 \cap V_2 = \{0\}$, by trigonometry, there is a constant $C_1$ independent of $x$, such that 
\begin{equation}
\label{eq;trig} |x(t_2) -x(t_1)| \ge \frac{1}{C_1}|x(t_1)|.
\end{equation}
Then for any $t \in [t_1, t_2]$, by the Cauchy-Schwarz inequality,  
$$ \begin{aligned}
|x(t)| & \le |x(t_1)| + |x(t)-x(t_1)| \\
& \le C_1|x(t_2)-x(t_1)| + |x(t)-x(t_1)| \\
       & \le (1 + C_1) \int_{t_1}^{t_2} |\dot{x}| \,dt \\
       & \le (1 +C_1) \sqrt{t_2 -t_1} \left( \int_{t_1}^{t_2} |\dot{x}|^2 \,dt \right)^{\ey} 
\end{aligned}
$$
As a result, 
\begin{equation}
\label{eq;coer} \begin{aligned}
\|x\|^2_{H^1} & = \int_{t_1}^{t_2} |x|^2 \,dt + \int_{t_1}^{t_2} |\dot{x}|^2 \,dt \\
& \le \left( 1 + (t_2 -t_1)^2 (1 +C_1)^2  \right) \int_{t_1}^{t_2} |\dot{x}|^2 \,dt \\
& \le 2\left( 1 + (t_2 -t_1)^2 (1 +C_1)^2  \right) \A(x; t_1, t_2).
\end{aligned}
\end{equation}
This proves the coercivity of $\A$ in $B_{t_1, t_2}(V_1, V_2)$. 

Now let $x|_{[t_1, t_2]}$ be a minimizer of $\A$. By Theorem \ref{thm;Marchal-rest-2},  for all $t \in [t_1, t_2]$, $x(t)$ is a collision-free path hitting $V_1$ and $V_2$ perpendicularly at the end points. As a result, it can be extended to an entire solution $x|_{\rr}$ of \eqref{eq;RNBP1} through the reflections with respect to $V_1$ and $V_2$.
\end{proof}

\begin{proof}[Proof of Theorem \ref{thm;perodic-2}]

Again for the existence of a minimizer, it is enough to prove the coercivity of $\A$ in $B_{t_1, t_2}(\Om_i(V_1), V_2)$.

If $\Om_i(V_1)$ is unbounded, to satisfy the assumption, $\Om_i(V_1)$ can not entirely be contained in $V_2$. Then \eqref{eq;trig} still holds, for all $x \in B_{t_1, t_2}(\Om_i(V_1), V_2)$ and some constant $C_1>0$ independent of $x$, and the proof of coercivity is exactly the same as above. 

If $\Om_i(V_1)$ is bounded, then for any $x \in B_{t_1, t_2}(\Om_i(V_1), V_2)$ there is a constant $C_1>0$ independent of $x$, such that for any $t \in [t_1, t_2]$,
$$ \begin{aligned}
|x(t)| & \le |x(t_1)| + |x(t) -x(t_1)| \\
       & \le C_1 + \int_{t_1}^{t_2} |\dot{x}| \,dt \\
       & \le C_1 + \sqrt{t_2 -t_1} \left( \int_{t_1}^{t_2} |\dot{x}|^2 \,dt \right)^{\ey}  
\end{aligned} 
$$
The coercivity then follows from a similar estimate as given in \eqref{eq;coer}.

Let $x|_{[t_1, t_2]}$ be a global minimizer of $\A$ in $B_{t_1, t_2}(\Om_i(V_1), V_2)$. First by Theorem \ref{thm;Marchal-restricted}, $x(t)$ is collision-free, for all $t \in (t_1, t_2)$. Meanwhile as $V_2$ is a proper linear subspace of $\rr^d$, by property (a) in Theorem \ref{thm;Marchal-rest-2}, $x(t_2)$ is also collision-free with $\dot{x}(t_2^-)$ perpendicular to $V_2$. 

On the other hand, as $\Om_i(V_1)$ is either a closed straight line segment or a closed ray, by property (b) and (d) in Theorem \ref{thm;Marchal-rest-2}, either $x(t_1)$ is collision-free with $\dot{x}(t_1^+)$ perpendicular to $\Om_i(V_1)$, or it collides with the one of the endpoints of $\Om_i(V_2)$ (there is only one endpoint if it is a ray). 

Moreover when $x(t_1)$ is a collision, the one sided limit given in properties (b) and (d) in Theorem \ref{thm;Marchal-rest-2} implies that after extending $x|_{[t_1, t_2]}$ to $x|_{\rr}$ through symmetries all the collisions must be regularizable.     
\end{proof}

\begin{proof}[Proof of Theorem \ref{thm;periodic-3}]
The proof of coercivity of $\A$ in $B_{t_1, t_2}(\Om_i(V_1), \Om_j(V_2))$ is almost the same as the proof in Theorem \ref{thm;perodic-2}. 
We will not repeat the details here. 

Let $x|_{[t_1, t_2]}$ be a global minimizer of $\A$ in $B_{t_1, t_2}(\Om_i(V_1), \Om_j(V_2))$. Again by Theorem \ref{thm;Marchal-restricted}, $x(t)$ is collision-free, for all $t \in (t_1, t_2)$. For $x(t_1)$, the proof is the same as in Theorem \ref{thm;perodic-2}. 

For $x(t_2)$, as $\Om_j(V_2)$ is either a closed straight line segment or a closed ray, by properties (c) and (e) in Theorem \ref{thm;Marchal-rest-2}, either $x(t_2)$ is collision-free with $\dot{x}(t_2^-)$ perpendicular to $\Om_j(V_2)$, or it collides with one of the endpoints of $\Om_j(V_2)$ (there is only one endpoint if it is a ray). 

Again, despite the possible collisions at $x(t_1)$ and $x(t_2)$, the results of the one-sided limits in Theorem \ref{thm;Marchal-rest-2} imply that after extending $x|_{[t_1, t_2]}$ to $x|_{\rr}$ through symmetries all the collisions must be regularizable.  
\end{proof}

Below we demonstrate some numerical examples with $V_1$ being the $x$-axis and $V_2$ a line through 
the origin. 
Retrograde orbits in figure~\ref{fig:retrograde1} and figure~\ref{fig:retrograde2} are examples from  \cite{Ch08,ChenLin09}. 
In  figure~\ref{fig:retrograde1}  and  figure~\ref{fig:euler}  $V_2$ is the $y$-axis. 
In  figure~\ref{fig:retrograde2}  $V_2$ forms a $\pi/3$ angle with $V_1$.

\begin{figure}[ht]
\begin{center}
\includegraphics[width=40mm]{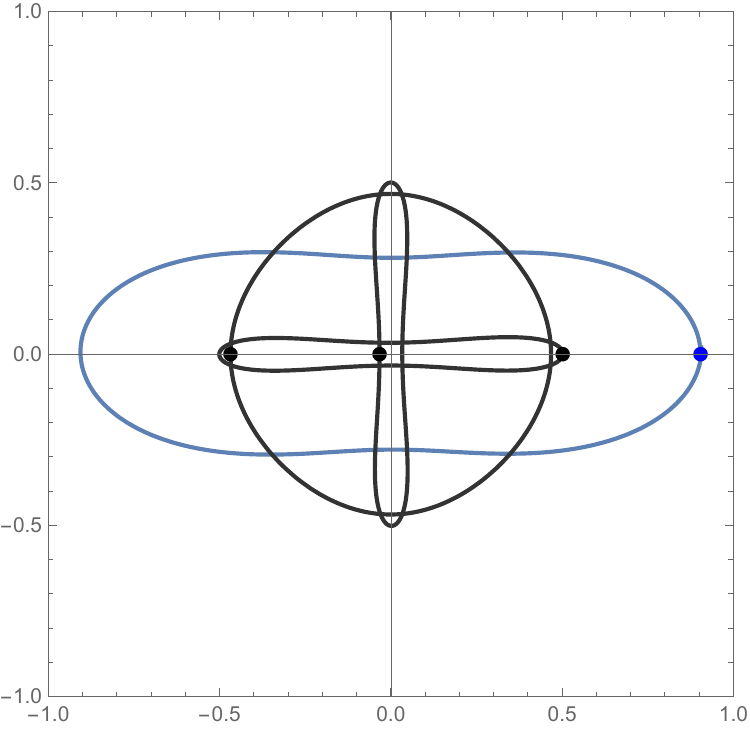} \quad 
\includegraphics[width=40mm]{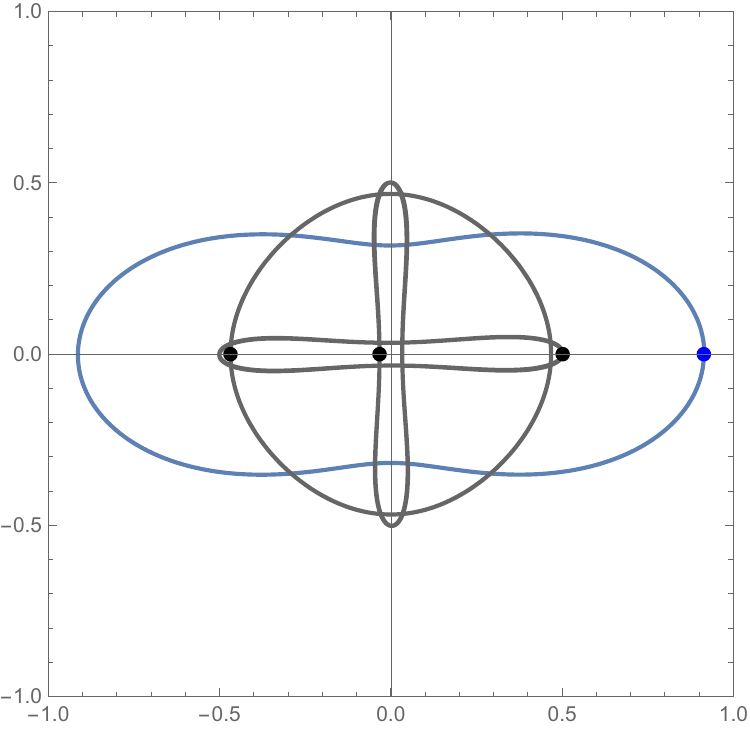} \quad 
\includegraphics[width=40mm]{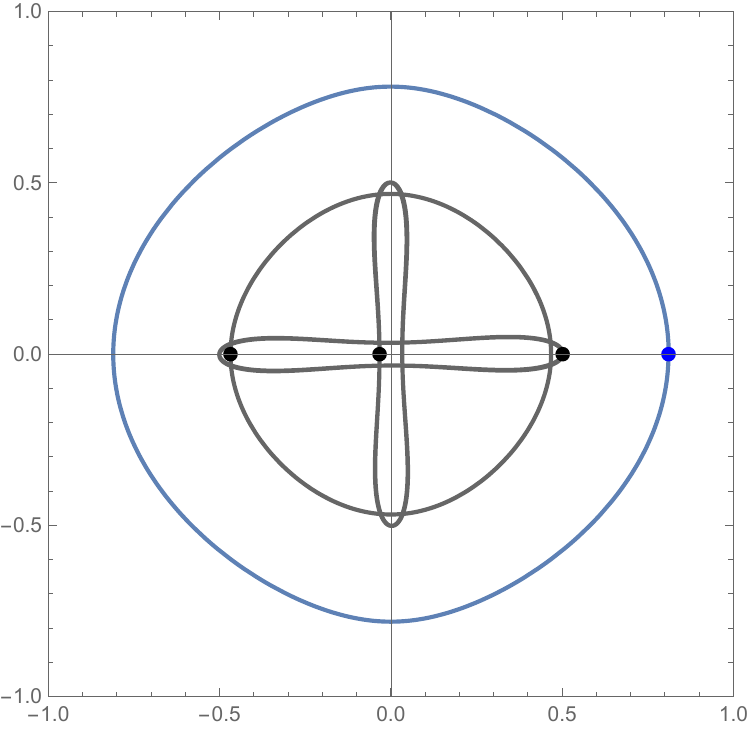} 
\caption{A retrograde solution of the 3-body problem plus a zero mass. }
\label{fig:retrograde1}     
\end{center}
\end{figure}

{\small
\begin{tabular} { l | l | l } 
\multicolumn{3}{l}{Initial Data of Primaries  } \\
\hline 
   $q_1(0)=(0.5014, 0)$ &  $q_2(0)=(-0.03310, 0)$ &  $q_3(0)=(-0.4683, 0)$    \\     
   $\dot{q}_1(0) = (0, -0.2858)$ & $\dot{q}_2(0) = (0, 1.8042)$ & $\dot{q}_3(0) = (0, -1.5184)$         \\
\hline 
\end{tabular}     \smallskip
 
\begin{tabular} {l} 
Initial Data of Zero Mass \\
\hline                 
(a) \;\;  $ x(0) = (0.905313, 0), \;\; \dot{x}(0) =  (0, 1.20303798)$ \\ \hline
(b) \;\;  $ x(0) = (0.91364, 0), \;\; \dot{x}(0) =  (0, -1.364363175)$ \\ \hline
(c) \;\;  $ x(0) = (0.81043, 0), \;\; \dot{x}(0) =  (0, -2.6976445)$ \\
\hline
\end{tabular}   
\captionof{table}{Numerical data for orbits in figure~\ref{fig:retrograde1}  with masses $(m_1,m_2,m_3)=(1,1,1)$. 
The accuracy of these data is limited by the data in supplements of \cite{ChenLin09}. } \label{table:example1}
}

\begin{figure}[ht]
\begin{center}
\includegraphics[width=40mm]{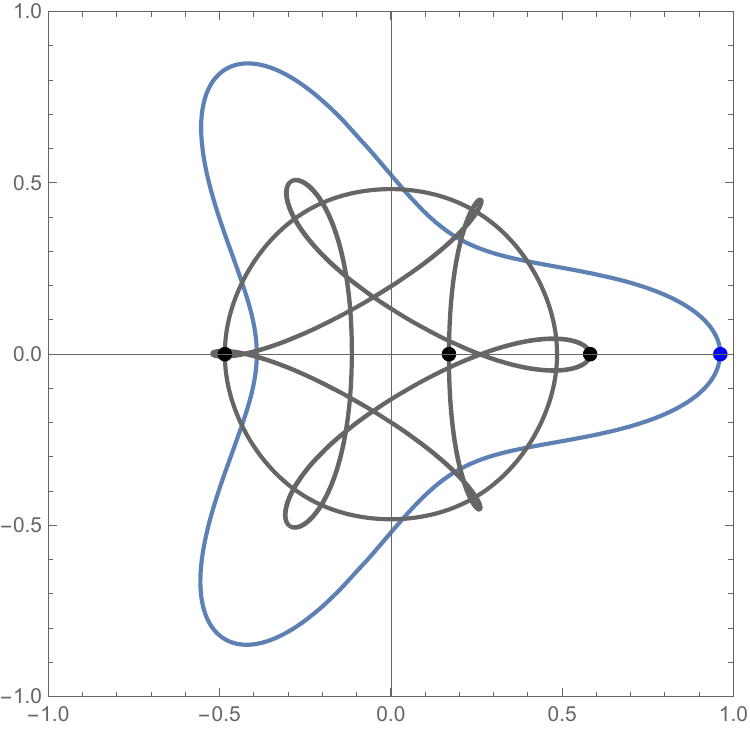} \quad 
\includegraphics[width=40mm]{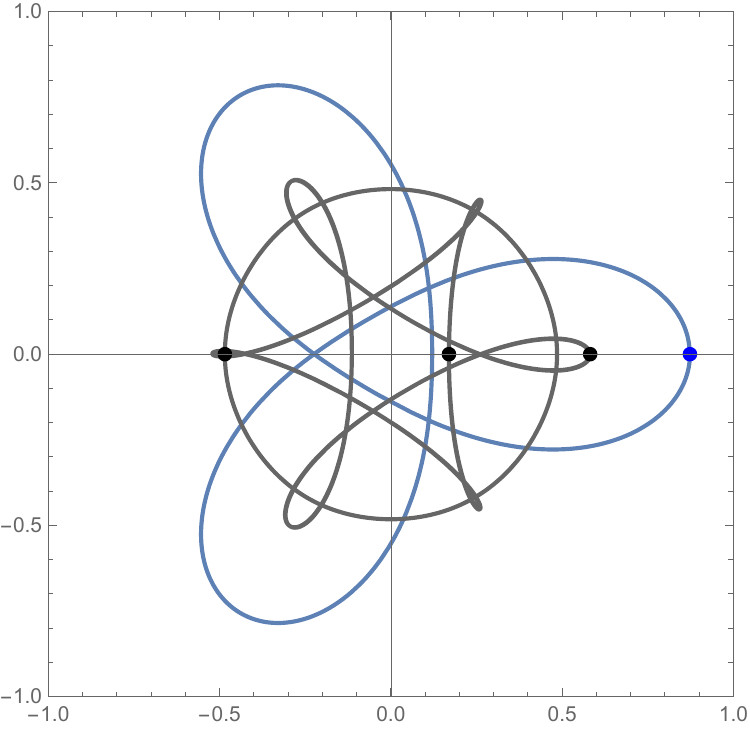} \quad 
\includegraphics[width=40mm]{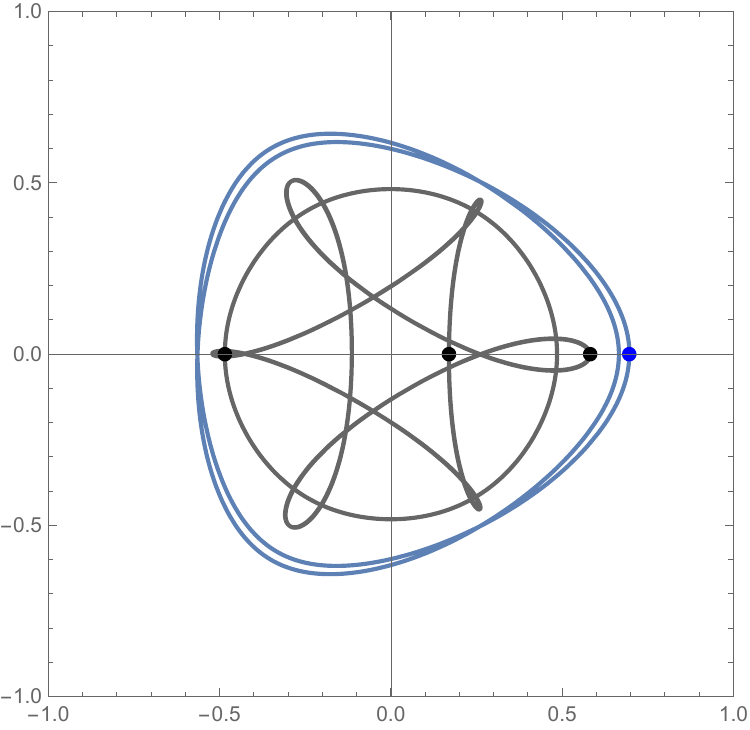} 
\caption{Another retrograde solution of the 3-body problem plus a zero mass. }
\label{fig:retrograde2}     
\end{center}
\end{figure}

{\small
\begin{tabular} { l | l | l } 
\multicolumn{3}{l}{Initial Data of Primaries } \\
\hline 
   $q_1(0)=(0.581737, 0)$ &  $q_2(0)=(0.169550, 0)$ &  $q_3(0)=(-0.484682, 0)$    \\     
   $\dot{q}_1(0) = (0, -0.361268)$ & $\dot{q}_2(0) = (0, 1.55210)$ & $\dot{q}_3(0) = (0, -1.02492)$         \\
\hline 
 \end{tabular}     \smallskip
 
\begin{tabular} {l} 
Initial Data of Zero Mass \\
\hline                 
(a) \;\;  $ x(0) = (0.9616965308, 0), \;\; \dot{x}(0) =  (0, 0.8154130153)$ \\ \hline
(b) \;\;  $ x(0) = (0.873058181, 0), \;\; \dot{x}(0) =  (0, -1.377300)$ \\ \hline
(c) \;\;  $ x(0) = (0.695917385, 0), \;\; \dot{x}(0) =  (0, -3.72196373248)$ \\
\hline
\end{tabular}     
\captionof{table}{Numerical data for orbits in figure~\ref{fig:retrograde2}  with masses $(m_1,m_2,m_3)=(0.6,0.8,1)$. 
Again, the accuracy of these data is limited by the data in supplements of \cite{ChenLin09}.} \label{table:example1}
}

\begin{figure}[ht]
\begin{center}
\includegraphics[width=35mm]{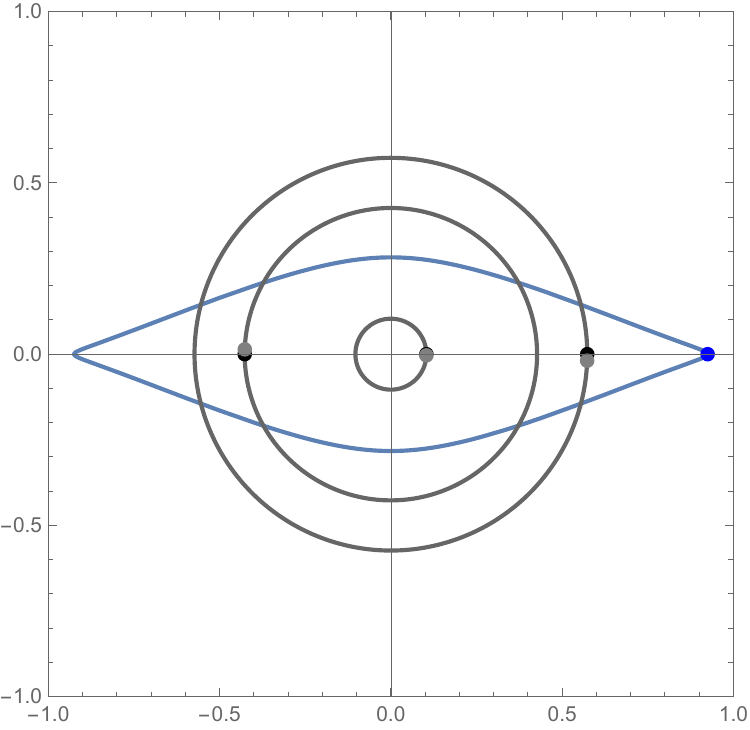} \;
\includegraphics[width=35mm]{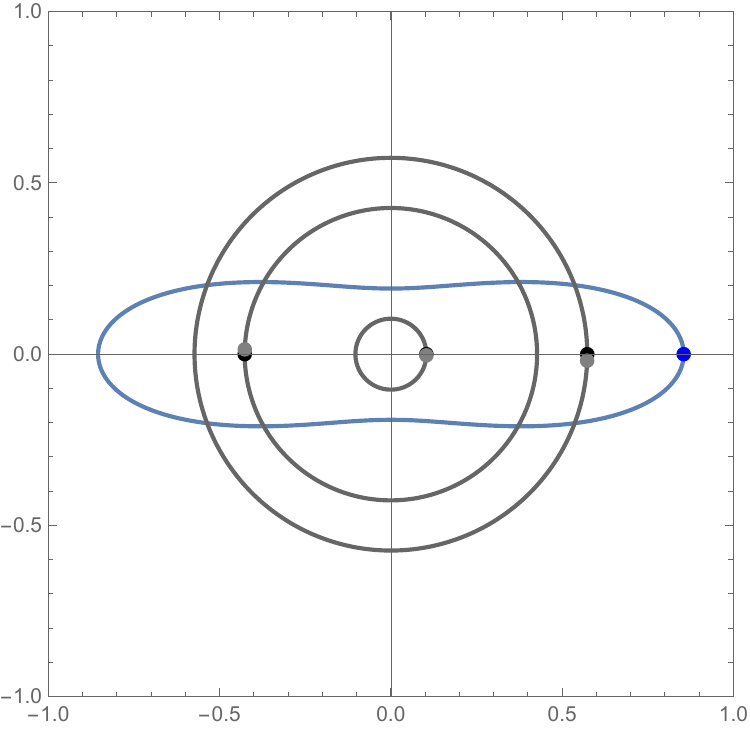} \;
\includegraphics[width=35mm]{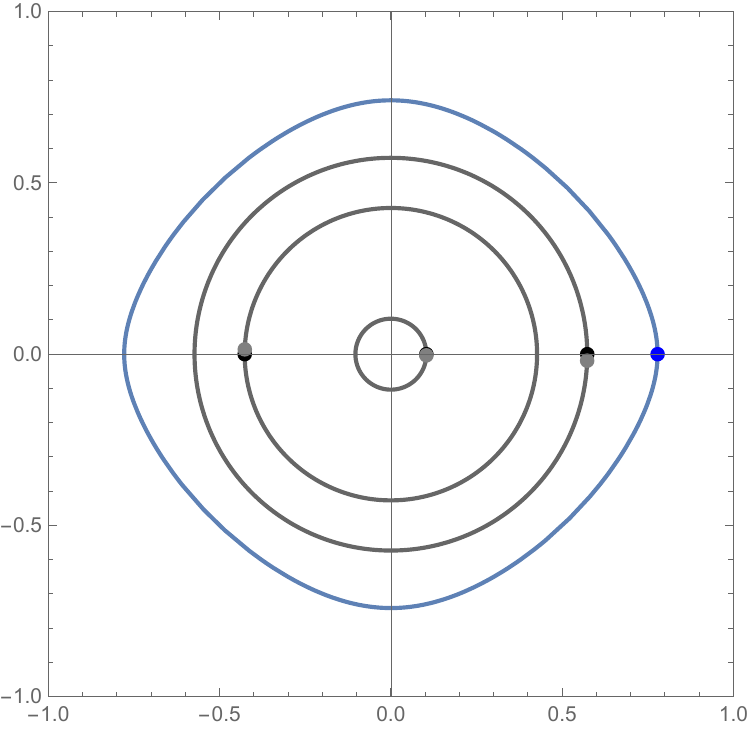} \;
\includegraphics[width=35mm]{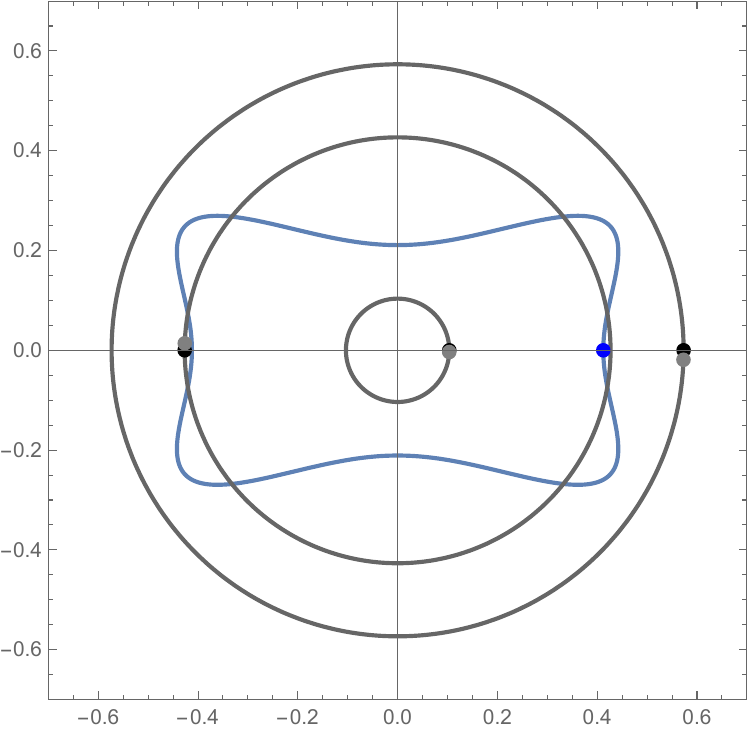} 
\caption{An Euler relative equilibrium solution of the 3-body problem plus a zero mass. }
\label{fig:euler}     
\end{center}
\end{figure}

{\small
\begin{tabular} { l | l | l } 
\multicolumn{3}{l}{Initial Data of Primaries } \\
\hline 
   $q_1(0)\!=\!(0.573225605518457, 0)$\! &  $q_2(0)\!=\!(0.10354878896308595, 0)$\! &  $q_3(0)\!=\!(-0.426774394481543, 0)$    \\     
   $\dot{q}_1(0) \!=\! (0, 1.6285108801044283)$\! & $\dot{q}_2(0) \!=\! (0, 0.29417794289824895)$\! & $\dot{q}_3(0) \!=\! (0, -1.2124488823812563)$         \\
\hline 
 \end{tabular}     \smallskip
 
\begin{tabular} {l} 
Initial Data of Zero Mass \\
\hline                 
(a) \;\;  $ x(0) = (0.92465649920, 0), \;\; \dot{x}(0) =  (0, 0.14802907824)$ \\ \hline
(b) \;\;  $ x(0) = (0.85474070604, 0), \;\; \dot{x}(0) =  (0, -1.0390306513997)$ \\ \hline
(c) \;\;  $ x(0) = (0.778532758929, 0), \;\; \dot{x}(0) =  (0, -2.3280140192)$ \\ \hline
(d) \;\;  $ x(0) = (0.4122450607052, 0), \;\; \dot{x}(0) =  (0, -2.0117809410548)$ \\ \hline
\end{tabular}     \smallskip
\captionof{table}{Numerical data for orbits in figure~\ref{fig:euler}  with masses $(m_1,m_2,m_3)=(0.6,0.8,1)$.  
The data are accurate to the 10th decimal places. } \label{table:example1}
}





\bibliographystyle{plain}
\bibliography{RefRestricted}



\end{document}